\documentclass[12pt]{amsart}
\usepackage{amsmath,bbm, color}
\usepackage{latexsym,mathrsfs,bm}
\usepackage{url}
\usepackage[english]{babel}
\usepackage{paralist}
\usepackage{amsthm,amsfonts,amssymb,amsmath}
\usepackage[latin1]{inputenc}

\newcommand{\tRe}{\textup{Re }}

\newcommand{\sumstar}{\sideset{}{^*}\sum}

\newcommand{\sumsharp}{\sideset{}{^\#}\sum}

\newcommand{\sumt}{\sideset{}{^{rel}}\sum}
\newcommand{\eps}{\varepsilon}

\newcommand{\A}{\mathcal{A}}

\newcommand{\es}[1]{\begin{equation}\begin{split}#1\end{split}\end{equation}}
\newcommand{\est}[1]{\begin{equation*}\begin{split}#1\end{split}\end{equation*}}

\renewcommand{\mod}[1]{~\pr{\textnormal{mod}~#1}}
\newtheorem{thm}{Theorem}[section]

\newtheorem{prop}[thm]{Proposition}
\newtheorem{lem}[thm]{Lemma}

\theoremstyle{remark}
\newtheorem{rem}{Remark}

\newtheorem{rem*}{Remark}
\newcommand{\pr}[1]{\left( #1\right)}
\newcommand{\pg}[1]{\left\{ #1\right\}}

\newcommand{\sgn}{\operatorname{sgn}}
\newcommand{\e}[1]{\operatorname{e}\pr{ #1}}

\newcommand{\bfrac}[2]{\left(\frac{#1}{#2}\right)}

\newcommand{\fG}{\mathfrak{G}}
\newcommand{\fI}{\mathfrak{I}}
\newcommand{\fS}{\mathfrak{S}}

\newcommand{\floor}[1]{\lfloor #1 \rfloor}

\newcommand{\cG}{\mathcal{G}}
\newcommand{\cJ}{\mathcal{J}}

\newcommand{\cR}{\mathcal{R}}

\def\sumstar{\operatornamewithlimits{\sum\nolimits^*}}

\newcommand{\sumtwo}{\operatorname*{\sum\sum}}
\newcommand{\sumthree}{\operatorname*{\sum\sum\sum}}

\newcommand{\comment}[1]{}

\let\originalleft\left
\let\originalright\right
\renewcommand{\left}{\mathopen{}\mathclose\bgroup\originalleft}
\renewcommand{\right}{\aftergroup\egroup\originalright}

\numberwithin{equation}{section}

\begin{document}
\title{The second moment of $GL(4) \times GL(2)$ $L$-functions at special points}

\date{\today}
\subjclass[2010]{11F11, 11F41, 11F72}
\keywords{Hecke Maass form, moments}

\author[V. Chandee]{Vorrapan Chandee}
\address{Department of Mathematics \\  Kansas State University \\
	138 Cardwell Hall, Manhattan, KS 66506, United States }
\email{chandee@ksu.edu }

\author[X. Li]{Xiannan Li}
\address{Department of Mathematics \\ Kansas State University \\
	138 Cardwell Hall, Manhattan, KS 66506, United States}
\email{xiannan@math.ksu.edu}

\allowdisplaybreaks
\numberwithin{equation}{section}
\selectlanguage{english}
\begin{abstract}
In this paper, we obtain upper bounds for the second moment of $L(u_j \times \phi, \frac{1}{2} + it_j)$, where $\phi$ is a Hecke Maass form for $SL(4, \mathbb Z)$, and $u_j$ is taken from an orthonormal basis of Hecke-Maass forms on $SL(2, \mathbb{Z})$ with eigenvalue $1/4 + t_j^2$. The bounds are consistent with the Lindel\"{o}f hypothesis. Previously these types of upper bounds are available for only $GL(n) \times GL(2)$, where $n \leq 3.$ 
\end{abstract}

\maketitle  
\section{Introduction}   
Statistical distribution of the values of $L$-functions are fundamentally related to interesting arithmetic objects.  Here, we are interested in understanding a $GL(4) \times GL(2)$ Rankin-Selberg $L$-functions at special points.  To be precise, we shall study the second moment of $L(u_j \times \phi, \frac{1}{2} + it_j)$, where $\phi$ is a Hecke Maass form for $SL(4, \mathbb Z)$ of type $(\nu_1, \nu_2, \nu_3)$, and $u_j$ is taken from an orthonormal basis of Hecke-Maass forms on $SL(2, \mathbb{Z})$ with eigenvalue $1/4 + t_j^2$.  

The point $\frac{1}{2} + it_j$ is a zero of the Selberg zeta function, and the question of how certain $L$-functions behave at these points appears in the work of Phillips and Sarnak \cite{PS} on deformation of cusp forms.  These points are also distinguished in the analytic theory of $L$-functions for leading to conductor dropping.  The latter phenomenon has made subconvex results at these points quite difficult to achieve through moments computation.  

In particular, the conductor dropping necessitates understanding higher moments when deriving subconvexity results while making high moments apparently more achievable.  The latter statement should be viewed critically - in particular, the study of such families introduces unfamiliar and delicate problems.  It is of great interest to understand how the difficulty scales.  In this direction, Luo \cite{Luo} showed that
\begin{align*}
\sum_{t_j \leq T} \left| L(u_j, \frac{1}{2} + it_j)\right|^6 \ll T^{9/4+\epsilon} \textup{, and}\\
\sum_{t_j \leq T} \left| L(u_j, \frac{1}{2} + it_j)\right|^8 \ll T^{5/2+\epsilon}.
\end{align*}
This should be compared to the conjectured optimal bound of
$$\sum_{t_j \leq T} \left| L(u_j, \frac{1}{2} + it_j)\right|^{2k} \ll T^{2+\epsilon},$$  for all $k\ge 0$.  Achieving the above bound for all $k$ implies the deep Lindel\"{o}f hypothesis, while subconvexity follows if the bound is achieved for any $k>4$.  Luo's work was based on a large sieve type inequality he derived for this family and the power loss in Luo's result  arises from suboptimal large sieve bound.  This is a superficial indication of deeper complexities in this $GL(2)$ family which are still not well understood.  Later, Young showed that
\begin{align*}
\sum_{t_j \leq T} \left| L(u_j, \frac{1}{2} + it_j)\right|^6 \ll T^{2+\epsilon},
\end{align*}thereby achieving essentially the optimal bound for the sixth moment.  This resulted not from general improvements to the large sieve bound, but rather the use of Fourier analytic techniques which use specific information about the coefficients involved.  It should be noted that the Young's sixth moment result above translates without change to give the same quality bound for the second moment of a $GL(3)\times GL(2)$ family.  Specifically, Young \cite{Young} proved that
$$
\sum_{t_j \leq T} \left| L(F \times u_j, \frac{1}{2} + it_j)\right|^2 \ll T^{2+\epsilon},
$$for $F$ a fixed $GL(3)$ Hecke Maass form.

Here, our focus shall be on the second moment of $GL(4) \times GL(2)$ $L$-functions.  To be precise, the main object of this paper is to prove the following theorem. 
\begin{thm} \label{thm:mainboundforGL4GL2} With notations as above, we have
	$$ \sum_{t_j \leq T} \left| L(u_j \times \phi, \frac{1}{2} + it_j)\right|^2 \ll T^{2 + \epsilon}. $$
\end{thm}

The techniques underlying the proof of our Theorem \ref{thm:mainboundforGL4GL2} can be modified to give the following bound on the eighth moment of the $GL(2)$ $L$-functions:

$$ \sum_{t_j \leq T} \left| L(u_j, \frac{1}{2} + it_j)\right|^8 \ll T^{2 + \epsilon}. $$

As mentioned before, these bounds are consistent with the Lindel\"{o}f hypothesis. Recently, the authors proved an analogous result for the eighth moment of the family of holomorphic modular forms with respect to the congruence subgroup $\Gamma_1(q)$ in \cite{CL3}.  Of course, the basic structure of the two families are radically different.  

In \S \ref{sec:initialsetup} to \S \ref{sec:IRKinitial} we reduce the moment problem in Theorem \ref{thm:mainboundforGL4GL2} to the proof of Lemma \ref{lem:mainlem2}, which is the main innovation in this paper.  We provide a sketch of the proof of Lemma \ref{lem:mainlem2} in \S \ref{subsec:sketch}.  

The initial reduction in \S \ref{sec:initialsetup} to \S \ref{sec:IRKinitial} is unexpectedly complicated by the unwieldly Hecke relations for $GL(4)$ and the need for bounds of the form $\sum_{\ell \sim L} |A(1, \ell, 1)|^2 \ll L$.  Interestingly, note that the latter bound depends not only on Rankin-Selberg theory, but also the work of Kim on functoriality of the exterior square on $GL(4)$ \cite{Kim}.


\section{Initial Setup} \label{sec:initialsetup}

We begin by introducing more precise notation.  As in definition 9.4.3 in \cite{Goldfeld}, for $\tRe s>1$, the Godement-Jacquet $L$-function associated to $\phi$ is 
$$ L(\phi, s) = \sum_{n = 1}^\infty \frac{A( 1, 1, n)}{n^s} = \prod_p \left( 1 - \frac{A(1, 1, p)}{p^s} + \frac{A(1, p, 1)}{p^{2s}} - \frac{A(p, 1, 1)}{p^{3s}} + \frac{1}{p^{4s}} \right)^{-1}.$$
Here, we normalize the coefficients $A(m_1, m_2, m_3)$ by setting $A(1, 1, 1) = 1.$ The dual Maass form $\widetilde{\phi}$ is of type $(\nu_3, \nu_2, \nu_1)$ and the  Fourier coefficient is  $\overline{A(m_1, m_2, m_3)} = A(m_3, m_2, m_1)$. Hence 
$$  L(\widetilde{ \phi}, s) = \sum_{n = 1}^\infty \frac{A( n,1,1)}{n^s} = \prod_p \left( 1 - \frac{A(p, 1, 1)}{p^s} + \frac{A(1, p, 1)}{p^{2s}} - \frac{A(1, 1, p)}{p^{3s}} + \frac{1}{p^{4s}} \right)^{-1}. $$ 
Now let
$$ G_{\nu_1, \nu_2, \nu_3}(s) = \pi^{-2s} \prod_{i = 1}^4 \Gamma\left( \frac{s - \alpha_i}{2}\right),$$
where 
\es{\label{def:alphai}\alpha_1 &= \frac{3}{2} - \nu_1 - 2\nu_2 - 3\nu_3 \\
	\alpha_2 &= -\frac 32 + 3\nu_1 + 2\nu_2 + \nu_3 \\
	\alpha_3 &= -\frac 12 - \nu_1 + 2\nu_2 + \nu_3 \\
	\alpha_4 &= \frac 12 - \nu_1 - 2\nu_2 + \nu_3. }
From Theorem 10.8.6 in \cite{Goldfeld}, the functional equation for $L(\phi, s)$ is
$$ G_{\nu_1, \nu_2, \nu_3}(s)L(\phi, s)  =   G_{\nu_3, \nu_2, \nu_1}(1-s)L(\widetilde{ \phi}, 1-s)$$

Recall that $(u_j)$ forms an orthonormal basis of Hecke-Maass cusp forms on $SL(2, \mathbb Z)$ with corresponding Laplace eigenvalues $1/4 + t_j^2$.  As usual, we write $\lambda_j(n)$ for the Hecke eigenvalue of the $n^{th}$ Hecke operator of the form $u_j$. For $\tRe s> 1$, $L(u_j, s)$ is given by
$$ L(u_j, s) = \sum_{n = 1}^\infty \frac{\lambda_j(n)}{n^s}.$$
Following the notation in Chapter 12.3 in \cite{Goldfeld}, the Rankin-Selberg $L$-function is defined by
\es{\label{def:Lfnc} L(u_j \times \phi, s) = \sumtwo_{m, n \geq 1}   \frac{\lambda_j(n) A(1,m, n)}{m^{2s}n^s} := \sum_{n = 1}^{\infty} \frac{\mathfrak c(n, u_j \times \phi )}{n^s}.}
For $u_j$ even, the completed $L$-function associated to $L(u_j \times \phi, s)$ is 
\est{\label{eqn:Lambdadef}\Lambda(u_j \times \phi, s) = \gamma(s, u_j \times \phi)    L(u_j \times \phi, s) ,}
where
\es{\label{def:gamma}\gamma(s, u_j \times \phi) := \pi^{-4s} \prod_{i = 1}^4 \Gamma\pr{\frac{s - it_j - \alpha_i}{2}} 
	\Gamma\pr{\frac{s + it_j - \alpha_i}{2}}.}
By Theorem 12.3.6 in \cite{Goldfeld}, the completed $L$-functions satisfies the following functional equation 
$$ \Lambda(u_j \times \phi, s) = \Lambda(u_j \times \widetilde{ \phi}, 1-s).$$

The case where $u_j$ is odd is very similar, the only change being a change in the argument of the Gamma factors appearing in Equation \eqref{eqn:Lambdadef} by $1/2$ which corresponds to the change in the Archimedean factors for $L(s, u_j)$.  We refer the reader to Chapter 12 of \cite{Goldfeld} for more details, and specifically to Section 12.3.6 for the functional equation.  

We start by reducing our second moment to an appropriate form.  Our setup in this section is very similar to the initial setup in Young's work \cite{Young} and we refer to Section 4  - 7 of \cite{Young} for a more detailed exposition.

We will first apply the approximate function equation (e.g. Theorem 5.3 in the book \cite{IK}). Define a smooth function
\est{V\pr{y, s} = \frac{1}{2\pi i} \int_{(3)} y^{-z} \frac{\gamma\left( s + z, u_j \times \phi \right)}{\gamma\pr{s, u_j \times \phi}} \frac{\mathcal H(z)}{z} \>dz, }
where $\gamma(s, u_j \times \phi)$ is defined in Equation (\ref{def:gamma}) and $\mathcal H(z)$ is an entire function with rapid decay along vertical lines. 
Moreover, let $\gamma^*(s, u_j \times \widetilde{ \phi})$ be a factor in the functional equation 
$$ \Lambda(u_j \times \widetilde{ \phi}, s) =  \gamma^*(s, u_j \times \widetilde{ \phi})L(u_j \times \widetilde{ \phi}, s),$$
and $V^*\left(y, s\right)$ is similar to $V(y, s)$ but replacing $\gamma(s, u_j \times \phi)$ with $\gamma^*(s, u_j \times \widetilde{ \phi}).$  Then the approximate functional equation gives that for any $Y > 0$
\est{ L\pr{u_j \times \phi, \frac{1}{2} + it_j } = \sum_{n} \frac{\mathfrak c(n, u_j \times \phi)}{n^{\frac{1}{2} + it_j}} V\left(\frac nY, \frac 12 + it_j\right) + \epsilon_j \sum_n \frac{\mathfrak c(n, u_j \times \widetilde{ \phi})}{n^{\frac{1}{2} - it_j}} V^*\left( nY, \frac 12 - it_j\right),}
where $|\epsilon_j | = 1.$ Due to the Stirling's formula, the leading term of an asymptotic formula of $V\pr{y, \frac 12 + it_j}$  is
$$ V_1 \pr{ y, t_j} = \frac{1}{2\pi i} \int_{(3)} \pr{\frac{t_j^2}{y}}^{z} \frac{h_0(z)}{z} \> ds$$
where $h_0(z)$ is a holomorphic function with exponential decay as $\textrm{Im}(z) \rightarrow \infty$. We also define $V_2(y,t_j)$ to be the corresponding leading term of an asymptotic formula of $V^*$, and $V_2$ has the same form as $V_1.$ 

It suffices to consider the leading terms of $V$ and $V^*$ in what follows.  Since there are $\ll \log T$ dyadic intervals up to $T^{2+\epsilon}$, we further restrict $t_j \sim T$ and  $n \sim N$ where $a \sim A$ is shorthand for $A < a \leq 2A$.  For simplicity, we consider only even forms $u_j$ - the difference for odd forms is a minor difference in the functional equation.

After applying Cauchy-Schwarz inequality, these terms are
\est{ \sum_{\substack{t_j \sim T \\ u_j \ \textrm{even}}}  \left( \left| \sum_{n \sim P } \frac{\mathfrak c(n, u_j \times \phi)}{n^{\frac 12 + it_j}} w_1 \pr{\frac nP} V_1\pr{\frac nY, t_j}  \right|^2  + \left| \sum_{n \sim P } \frac{\mathfrak c(n, u_j \times \widetilde{ \phi})}{n^{\frac 12 - it_j}} w_1 \pr{\frac nP} V_1\pr{ nY, t_j}  \right|^2  \right),}
where $w_1$ is a compactly supported smooth function which gives a smooth partition of unity, i.e. $\sum_{j \geq 1} w_1 \pr{\frac{t}{2^j}} = 1.$ By the change of variables and combining smooth functions, we need to consider
\est{ \sum_{\substack{t_j \sim T \\ u_j \ \textrm{even}}} \left| \sum_{n \sim P } \frac{\mathfrak c(n, u_j \times \phi)}{n^{\frac 12 + it_j}} w_2 \pr{\frac nP}   \right|^2 }  
where $P \ll T^{2 + \epsilon}$ and $w_2$ has the same properties as $w_1$. 

Next we will add the weight $\alpha_j = |\rho_j(1)|^2 w(t_j)$, where $\rho_j(1)$ is the Fourier coefficient of the $u_j(z)$, and the nonnegative smooth function $w(t_j)$ is defined by
$$ w(t_j) = 2 \frac{\sinh \pr{\pr{\pi - \frac 1T}t_j}}{\sinh \pr{2\pi t_j}}. $$
We obtain that $ \alpha_j \sim \frac{|\rho_j(1)|^2}{\cosh(\pi t_j) } \exp(- t_j/T)$ when $t_j \sim T$, and $t_j^{-\epsilon} \ll \alpha_j \ll t_j^\epsilon$ (see e.g. (5.4) and (5.5) in \cite{Young} ). The introduction of this weight normalizes our forms for more convenient application of Kuznetsov's formula. Also, we remove the conditions $t_j \sim T$ and $u_j$ is even, allowed by positivity of these terms.  Hence it is enough to show that
\es{ \label{eqn:Hellsum}H := \sum_{t_j} w(t_j) |\rho_j(1)|^2 \left| \sum_{n \sim P } \frac{\mathfrak c(n, u_j \times \phi)}{n^{\frac 12 + it_j}} w_2 \pr{\frac nP}   \right|^2  \ll T^{2 + \epsilon}. } 

By Equation (\ref{def:Lfnc}), Cauchy-Schwarz inequality and $P \ll T^{2 + \epsilon}$, we have that 

\begin{equation*}
H \ll T^\epsilon \sum_{\ell \ll \sqrt P} \frac{H_\ell}{\ell} 
\end{equation*}

where 
$$ H_\ell = \sum_{t_j} w(t_j) |\rho_j(1)|^2 \left| \sum_{n \sim N} w_2\left(\frac n N \right) \frac{A(1, \ell,n ) \lambda_j(n)}{n^{1/2 + it_j}} \right|^2,  $$
and $N = P/\ell^2 $. To prove Equation (\ref{eqn:Hellsum}), it suffices to show the following two results.

\begin{prop} \label{prop:mainprop} With the above notations and  $N = \frac{P}{\ell^2} \ll \frac{T^{2 + \epsilon}}{\ell^2}$, we have
	$$ H_{\ell} \ll T^{2 + \epsilon} \left( 1 + \sum_{n \sim N} \frac{|A(1, \ell, n)|^2}{n}\right). $$
\end{prop}

\begin{lem} \label{lem:sum} Let $K, L$ be fixed.  Then
\begin{equation}
\sum_{k \sim K}\sum_{\ell \sim L} |A(k, \ell, 1)|^2 \ll (KL)^{1+\epsilon}.
\end{equation}
\end{lem}

Thus from Proposition \ref{prop:mainprop} and Lemmma \ref{lem:sum}, we have that 
\est{ H \ll T^{2 + \epsilon} \sum_{ \ell \leq \sqrt P} \sum_{n \sim \frac P{\ell^2}} \frac{|A(1, \ell, n)|^2}{n\ell} \ll T^{2 + \epsilon}}
as desired. 

The proof of Lemma \ref{lem:sum} is more involved than might appear at first sight, and we will prove it in \S \ref{sec:lemsumpf}.  The proof of the main Proposition \ref{prop:mainprop} will be the focus of the rest of the paper.
\\

{\bf Remarks on notation.} As usual, we will use $\epsilon$ to denote a small positive real number, not necessarily the same at each occurrence. At some point in the
paper a parameter $ \epsilon_1 > 0$ will be fixed once it is introduced and be chosen later. Moreover $\e{x} = \exp(2\pi i x)$.

\section{Ramanujan on average} \label{sec:lemsumpf}

We first prove Lemma \ref{lem:sum}, beginning with a Mobius inversion type result which was used by Xiaoqing Li and Young in \cite{LiYoung} in an analogous $GL(3)$ case.  We will show that
\begin{equation}\label{eqn:mobinversion}
\sum_{\substack{d |(k, \ell)\\ e|(d, k/d)}} \mu(d) \mu(e) A\left(\frac{k}{de}, 1, 1\right) A\left(1, \frac \ell d, \frac de\right) = A(k, \ell, 1).
\end{equation}
Indeed, the Hecke relations give (for instance, see Theorem 9.3.11 of \cite{Goldfeld})
\begin{equation} \label{heckerel}
A(k, 1, 1) A(1, \ell, d) = \sum_{\substack{c_1c_2|k\\c_1|\ell,\  c_2|d}} A\left(\frac{k}{c_1c_2}, \frac {\ell}{c_1}, \frac{dc_1}{c_2}  \right),
\end{equation}from which \eqref{eqn:mobinversion} follows by standard Mobius inversion type manipulations via
\begin{align*}
\sum_{\substack{d |(k, \ell)\\ e|(d, k/d)}} \mu(d) \mu(e) A\left(\frac{k}{de}, 1, 1\right) A\left(1, \frac \ell d, \frac de\right) &= \sum_{\substack{d|(k, \ell)\\e|(d, k/d)\\c_1|\ell/d, \ c_2|d/e\\c_1c_2|k/de}} \mu(d)\mu(e) A\left(\frac{k}{dec_1c_2}, \frac{\ell}{dc_1}, \frac{dc_1}{ec_2} \right)\\
&= \sum_{\substack{d|(k, g_1)\\g_1|\ell, \ g_2|d\\ g_1g_2|k}} \mu(d) A\left(\frac{k}{g_1g_2}, \frac{\ell}{g_1}, \frac{g_1}{g_2} \right) \sum_{e|g_2} \mu(e)\\
&= \sum_{g_1|(\ell, k)} A\left(\frac{k}{g_1}, \frac{\ell}{g_1}, g_1 \right) \sum_{d|g_1} \mu(d) = A(k, \ell , 1).
\end{align*}

Now, by \eqref{eqn:mobinversion} and Cauchy-Schwarz,
\begin{align*}
|A(k, \ell, 1)|^2 
&= \left|\sum_{\substack{d |(k, \ell)\\ e|(d, k/d)}} \mu(d) \mu(e) A\left(\frac{k}{de}, 1, 1\right) A\left(1, \frac \ell d, \frac de\right)\right|^2 \\
&\le \left(\sum_{\substack{d |(k, \ell)\\ e|(d, k/d)}} 1^2 \right) \left(\sum_{\substack{d|(k, \ell)\\ e|(d, k/d)}} |\mu(d)| \left|A\left(\frac{k}{de}, 1, 1\right)\right|^2  \left|A\left(1, \frac \ell d, \frac de\right)\right|^2 \right) \\
&\ll K^\epsilon \sum_{\substack{d|(k, \ell)\\ e|(d, k/d)}} |\mu(d)| \left|A\left(\frac{k}{de}, 1, 1\right)\right|^2 \left|A\left(1, \frac \ell d, \frac de\right)\right|^2.
\end{align*}
Hence,

\begin{align*}
\sum_{k \sim K}\sum_{\ell \sim L} |A(k, \ell, 1)|^2 
&\ll K^\epsilon \sum_{d\le \min(L, K)} |\mu(d)| \sum_{d = ef} \left(\sum_{k \sim K/de}\left|A\left(k, 1, 1\right)\right|^2 \sum_{\ell \sim L/d}    \left|A\left(1, \ell , f\right)\right|^2\right)\\
&\ll K^\epsilon \sum_{d\le \min(L, K)} |\mu(d)| \sum_{d = ef} \left(\frac{K}{de} \sum_{\ell \sim L/d}    \left|A\left(1, \ell , f\right)\right|^2\right)
\end{align*}since 
$$ \sum_{k \sim K/de}\left|A\left(k, 1, 1\right)\right|^2 \ll \frac{K}{de},
$$Note that in the sum above, for $f|d$, we may assume that $f$ is squarefree, due to the presence of the $|\mu(d)|$ factor.  Thus, Lemma \ref{lem:sum} follows from the following Lemma.

\begin{lem}\label{lem:lsumbdd}
	For any positive squarefree integers $a$ and $f$
	\begin{equation*}
	\sum_{\ell \sim L} |A(a, \ell , f)|^2 \ll afL.
	\end{equation*}	
\end{lem}

Before proving Lemma \ref{lem:lsumbdd}, we first state the following Lemma.

\begin{lem}\label{lem:adjsumbdd}
	\begin{equation*}
	\sum_{\ell\sim L} |A(1, \ell, 1)|^2 \ll L.
	\end{equation*}
\end{lem}
\begin{proof}
	We note that the exterior square $L$-function $L(s, \phi, \Lambda^2)$ has the Dirichlet series 
	\begin{equation}\label{eqn:exteriorsquareseries}
	L(s, \phi, \Lambda^2) = \sum_{\ell\ge 1} \frac{A(1, \ell, 1)}{\ell^s}.
	\end{equation}
Representations of such $L$-functions were studied by Bump and Friedberg \cite{BF} as well as Jacquet and Shalika \cite{JS}.  We refer to the later work of Kontorovich \cite{Kon} for a relatively elementary proof of \eqref{eqn:exteriorsquareseries} for $GL(n)$.   

Fortunately, we also know that the exterior square is essentially automorphic by the work of Kim.  To be precise, Kim's Theorem A \cite{Kim} tells us that there is an automorphic $L$-function on $GL(6)$ with the same Euler product as $L(s, \phi, \Lambda^2)$ except possibly the local factors at $2$ and $3$.  Standard contour integration of the Rankin-Selberg $L$-function (see Remark 12.1.8 \cite{Goldfeld}) then gives us the result of the Lemma.
\end{proof}

Now we turn to the proof of Lemma \ref{lem:lsumbdd}.  We proceed by induction on $af\floor{L} \ge 1$.  The base case $af = 1$ is covered by Lemma \ref{lem:adjsumbdd} (of course, this also covers the trivial case $af\floor{L} = 1$).  Thus we assume $af >1$ and since $a$ and $f$ are in symmetric positions, we may without loss of generality assume that $f>1$.  Then $p \| f$ for some prime $p$ since $f$ is squarefree.  By Hecke multiplicativity (Theorem 9.3.11 of \cite{Goldfeld}) in (\ref{heckerel}) again, we have
\begin{equation*}
A(a, \ell, fp^{-1}) A(1, 1, p) = \sum_{\substack{c_1 | a\\ c_2|\ell \\ c_1c_2|p}} A\left(\frac{ac_2}{c_1}, \frac{\ell}{c_2}, \frac{f}{c_1c_2} \right).
\end{equation*}Rearranging this, we have
\begin{align}\label{eqn:lsumbdd1}
\sum_{\ell \sim L} |A(a, \ell , f)|^2
&= \sum_{\ell \sim L} \left|A(a, \ell, f p^{-1}) A(1, 1, p) - \sum_{\substack{c_1 | a\\ c_2|\ell \\ c_1c_2=p}} A\left(\frac{ac_2}{c_1}, \frac{\ell}{c_2}, \frac{f}{c_1c_2} \right)\right|^2 \notag \\
&\le 48 \sum_{\ell\sim L} \left|A(a, \ell, f p^{-1})\right|^2  p^{(1-2/17)} + 3\sum_{\ell \sim L} \sum_{\substack{c_1 | a\\ c_2|\ell \\ c_1c_2=p}} \left| A\left(\frac{ac_2}{c_1}, \frac{\ell}{c_2}, \frac{f}{c_1c_2} \right)\right|^2,
\end{align}where we have applied Cauchy-Schwarz while noting that the original sum inside the absolute values is of length $3$ and have used the bound of Luo, Rudnick and Sarnak of $|A(1, 1, p)| \le 4 p^{1/2 - 1/17}$ (see Theorem 12.5.1 of Goldfeld's book \cite{Goldfeld}).  Applying the induction hypothesis, we see that the quantity on the right of \eqref{eqn:lsumbdd1} is

\begin{align}\label{eqn:lsumbdd2}
&48 \sum_{\ell\sim L} \left|A(a, \ell, f p^{-1})\right|^2  p^{(1-2/17)} + 3\sum_{\substack{c_1 | a\\ c_1c_2=p}} \sum_{\ell \sim L/c_2} \left| A\left(\frac{ac_2}{c_1}, \ell, \frac{f}{p} \right) \right|^2 \notag \\
&\ll 48p^{(-2/17)} afL +  3\sum_{\substack{c_1 | a\\ c_1c_2=p}} \frac{afL}{c_1 p} \notag \\
&= 3\frac{afL}{p^{2/17}} \left(16+ \frac{1}{p^{15/17}}\left(1+\frac{1}{p}\right)\right).
\end{align}

In applying our induction hypothesis, we have noted that $\frac{af}{p} \floor{L} < af \floor{L}$ and that $\frac{ac_2f}{c_1p} \floor{\frac{L}{c_2}} < af\floor{L} \frac{c_2}{c_1 p} \le  af\floor{L}$ for $L\ge 1$.  For all but finitely many primes $p$, 
\begin{equation}\label{eqn:lsumbdd3}
\frac{3}{p^{2/17}} \left(16+ \frac{1}{p^{15/17}}\left(1+\frac{1}{p}\right)\right) \le 1, 
\end{equation}
and we have proven the Lemma with no increase in the implied constant.
For those finite number of exceptions to \eqref{eqn:lsumbdd3}, the quantity in \eqref{eqn:lsumbdd2} is still
$$\ll afL,
$$which suffices for the Lemma.  Note that in these finite number of cases, the implied constant has increased - this is acceptable since $f$ can have only a bounded number of prime factors not satisfying \eqref{eqn:lsumbdd3}.

We now proceed to prove a similar result in Lemma \ref{lem:ramanujanonaverage} which will be useful later.  To do so, we first prove another Mobius inversion type result.
\begin{lem}\label{lem:faiprimepower}
	For any prime $p$ and non-negative integers $k, \ell, n$,
\begin{align*}
	A(p^k, 1, 1) A(1, p^\ell, p^n) &= A(p^k, p^\ell, p^n) + A(p^{k-1}, 1, 1) A(1, p^\ell, p^{n-1}) \\
	&+ A(p^{k-1}, 1, 1) A(1, p^{\ell-1}, p^{n+1}) - A(p^{k-2}, 1, 1)A(1, p^{\ell-1}, p^n),
\end{align*}
where we shall follow the convention that $A(a, b, c) = 0$ whenever one of $a, b$ or $c$ is not an integer.
\end{lem}
\begin{proof}
	We first show that
	\begin{equation}\label{eqn:multlem1}
	A(p^k, 1, 1) A(1, p^\ell, p^n) - A(p^k, p^\ell, p^n) - A(p^{k-1}, 1, 1) A(1, p^\ell, p^{n-1}) = \sum_{a|(p^{k-1}, p^{l-1})} A\left( \frac{p^{k-1}}{a}, \frac{p^{\ell-1}}{a}, p^{n+1}a\right).
	\end{equation}
	Indeed, by Hecke multiplicativity (Theorem 9.3.11 of \cite{Goldfeld})
	\begin{align}\label{eqn:multlem2}
	A(p^k, 1, 1) A(1, p^\ell, p^n) &= \sumtwo_{\substack{a, b,  \, ab|p^k\\a|p^\ell, \ b|p^n}} A\left(\frac{p^k}{ab}, \frac{p^\ell}{a}, \frac{p^n a}{b}\right)\notag \\
	&= \sum_{\substack{a|(p^k, p^\ell)}} A\left(\frac{p^k}{a}, \frac{p^\ell}{a}, p^n a\right) + \sumtwo_{\substack{a, b, \ ab|p^{k-1}\\a|p^\ell, \  b|p^{n-1}}} A\left(\frac{p^{k-1}}{ab}, \frac{p^\ell}{a}, \frac{p^{n-1} a}{b}\right) \notag \\
	&= \sum_{\substack{a|(p^k, p^\ell)}} A\left(\frac{p^k}{a}, \frac{p^\ell}{a}, p^n a\right) + A(p^{k-1}, 1, 1) A(1, p^{\ell}, p^{n-1}),
	\end{align}where the second line is simply splitting the previous sum into the two cases $b=1$ and $b\ge p$, and the last line is an application of Hecke multiplicativity.  Our claim in \eqref{eqn:multlem1} follows from this	and
	$$\sum_{\substack{a| (p^k, p^\ell)}} A\left(\frac{p^k}{a}, \frac{p^\ell}{a}, p^n a\right) = A(p^k, p^\ell, p^n) + \sum_{\substack{a| (p^{k-1}, p^{\ell-1})}} A\left(\frac{p^{k-1}}{a}, \frac{p^{\ell-1}}{a}, p^{n+1} a\right).
	$$
	
	On the other hand, by \eqref{eqn:multlem2} with $(k, \ell, n)$ replaced by $(k-1, \ell-1, n+1)$,
	\begin{equation}
	A(p^{k-1}, 1, 1) A(1, p^{\ell-1}, p^{n+1}) - A(p^{k-2}, 1, 1)A(1, p^{\ell-1}, p^n) = \sum_{a|(p^{k-1}, p^{l-1})} A\left(\frac{p^{k-1}}{a}, \frac{p^{\ell-1}}{a}, p^{n+1} a\right),
	\end{equation}
	and this proves the Lemma when combined with \eqref{eqn:multlem1}.
\end{proof}

This Lemma about prime powers leads directly to the following more general Lemma \ref{lem:faigeneral}, which follows by multiplicativity.  Of course, Lemma \ref{lem:faigeneral} can be proven by Mobius inversion as well at the cost of seeming less motivated.
\begin{lem}\label{lem:faigeneral}
	\begin{equation}
	A(k, \ell, n) = \sumthree_{\substack{d, e, f \\ d|(k, \ell), \ e|(d, k/d), \\ f|(k, n)}} \mu(df)\mu(e) A\left(\frac{k}{dfe}, 1, 1\right) A\left(1, \frac{\ell }{d}, \frac{dn}{ef}\right).
	\end{equation}
\end{lem}

Now we are ready to prove the following Lemma.
\begin{lem}\label{lem:ramanujanonaverage}
	For any $M>0$ and positive integers $b,c$,
	\begin{equation*}
		\sum_{m \le M} |A(m, b, c)|^2 \ll (bcM)^{1+\epsilon}.
	\end{equation*}	
\end{lem}
\begin{proof}
	We apply Lemma \ref{lem:faigeneral} and Cauchy-Schwarz to see that
	\begin{align}\label{eqn:ramlem1}
	\sum_{m \le M} |A(m, b, c)|^2 
	\le (bcM)^\epsilon \sumthree_{\substack{d, e, f \\ d|b, \ e|d,  \ f|c}} \left|A\left(1, \frac{b }{d}, \frac{dc}{ef}\right)\right|^2 \sum_{\substack{m\le M \\ d|m, e|m/d \\ f|m}} \mu^{2}(df) \left|A\left(\frac{m}{def}, 1, 1\right)\right|^2. 
	\end{align}
	Since $\mu^2(df) = 0$ if $(d, f) > 1$, we can replace $\mu^2(df)$ by the condition $(d,f) = 1.$  Note the conditions $d|m, e|m/d, f|m$ and $e|d$ yields that $def | m.$  When $A$ is the Fourier coefficient of $\phi$, then $A(1, b, c) = B(c, b, 1)$ where $B(c, b, 1)$ is the Fourier coefficient of the dual form of $\phi$, so Lemma \ref{lem:sum} implies that 
	\begin{equation}\label{eqn:ramlem2}
	 \left|A\left(1, \frac{b }{d}, \frac{dc}{ef}\right)\right|^2  \ll \left( \frac{bc}{ef}\right)^{1+\epsilon},
	\end{equation}simply by dropping all but one term in the sum, while standard contour integration of the Rankin-Selberg $L$-function gives
	\begin{align}\label{eqn:ramlem3}
	\sum_{\substack{m\le M \\ def| m}} \left|A\left(\frac{m}{dfe}, 1, 1\right)\right|^2
	&\le \sum_{m\le M/dfe} \left|A\left(m, 1, 1\right)\right|^2 \ll \frac{M}{dfe}.
	\end{align}
	Using \eqref{eqn:ramlem2} and \eqref{eqn:ramlem3} in \eqref{eqn:ramlem1} gives that
	\begin{equation*}
	\sum_{m \le M} |A(m, b, c)|^2  \ll (bcM)^\epsilon \sumthree_{\substack{d, e, f \\ d|b,\ e|d, \ f|c}} \frac{bcM }{de^2f^2} \ll (bcM)^{1+\epsilon},
	\end{equation*}as desired.
\end{proof}

\section{Setting up for Fourier analysis}

When $N \ll T$, Proposition \ref{prop:mainprop} follows immediately from an application of the large sieve type bound due to Luo (see Theorem 1 of \cite{Luo}).  For the reader's convenience, we state that bound here.  For any sequence of complex numbers $a_n$, we have that
\begin{equation}
\sum_{t_j \le T} (\cosh \pi t_j)^{-1} \left| \sum_{n\le N} a_n \rho_j(n)n^{it_j} \right|^2 \ll (T^2 + T^{3/2}N^{1/2} + N^{5/4}) (NT)^\epsilon \sum_{n\le N} |a_n|^2.
\end{equation}

Actually, the weaker bound claimed in (7) of Luo \cite{Luo} would suffice, but that bound is cited as Theorem 6 of \cite{DI} by Luo, and Theorem 6 of \cite{DI} required conditions on the coefficients which are not available for us.  We thank one of the anonymous referees for pointing this observation.

For $N \gg T$, we require the following theorem, which is Theorem 7.1 of Young's work \cite{Young}.

\begin{thm}  \label{thm:young7}Let 
	\est{ S(\mathcal A) = \sum_{t_j} w(t_j) |\rho_j(1)|^2 \left| \sum_{n \sim N} a_n \lambda_j(n) n^{it_j}\right|^2}
	For any $1 \leq X \leq T$ and $N \gg T, $ we have
	\est{ S(\mathcal A) = S_1(\mathcal A ; X) + O\pr{T^2 + \frac{NT}{X} + \frac{N^{\frac 32}}{T}} N^\epsilon \|\mathcal A\|^2}
	where $\|\mathcal A\|^2 = \sum_{n \sim N} |a_n|^2$, and
	\est{S_1(\mathcal A; X) \ll T \sum_{r < X} \frac{1}{r^2} \sum_{0 \neq |k| \ll rT^\epsilon} \int_{-T^{\epsilon}}^{T^{\epsilon}} \min \left\{ \frac{1}{|u|}, \frac{r/|k|}{1 + u^2}\right\} \left| \sum_n a_n S(k, n; r) \e{\frac{un}{rT}} \right|^2 \> du.}
	Here $S(k, n; r)$ is the usual Kloosterman sum defined by
	$$ S(k, n; r) = \sumstar_{x \mod r} \e{\frac{kx + n \overline{x}}{r}},$$
	where $\sumstar_{x \mod r}$ represents the summation restricted to coprime residue mod $r$. 
	
\end{thm}

\begin{rem}
	The upper bound for $ S_1(\mathcal A; X)$ stated in Theorem 7.1 of Young's work \cite{Young} is actually
	\est{S_1(\mathcal A; X) \ll T \sum_{r < X} \frac{1}{r} \sum_{0 \neq |k| \ll rT^\epsilon} \frac{1}{|k|} \int_{-T^{-\epsilon}}^{T^{-\epsilon}}  \left| \sum_n a_n S(k, n; r) \e{\frac{un}{rT}} \right|^2 \> du. }
	However this is not enough to obtain the bound $T^{2 + \epsilon}$ in Proposition \ref{prop:mainprop} especially for the case when $k \sim K$ is small and $r \sim R$ is big.  Specifically in that case, Poisson summation over $k$ leads to a long dual sum while the factor of $\frac{1}{|k|}$ is too large.  We instead keep the more flexible bound from Equation (8.8) of \cite{Young} and truncate the integral to $T^{\epsilon}$ in the same way as (8.9) and (8.12) of \cite{Young}.  
\end{rem}

After we apply Theorem \ref{thm:young7} to $\mathcal H_{\ell}$, 
$$ S(\A) - S_1(\A, X) \ll T^{2 + \epsilon} \left( \sum_{n \sim N} \frac{|A(1, \ell, n)|^2}{n}\right) $$
upon choosing $X = \min\{T, \frac NT\}$ and using $N  = \frac{P}{\ell^2}$ and $P \ll  T^{2 + \epsilon}.$ The above bound suffices to upon using Lemma \ref{lem:sum}.

Thus, we find that in order to bound $H_\ell$, we need to bound 
\est{&T \sum_{r < X} \frac{1}{r^2} \sum_{0 \neq |k| \ll rN^{\epsilon}} \int_{-T^{\epsilon}}^{T^{\epsilon}} \min \left\{ \frac{1}{|u|}, \frac{r/|k|}{1 + u^2}\right\}  \left| \frac{1}{\sqrt N}\sum_n A(1, \ell, n)S(k, n; r) w_3\left(\frac nN\right) e\pr{\frac{un}{rT}} \right|^2 \> du, }
where $w_3(x) = \frac{w_2(x)}{\sqrt x}$, $N  =  \frac{P}{\ell^2} \ll \frac{T^{2+\epsilon}}{\ell^2}$ and $ 1 \leq X = \min \{T, \frac{N}{T}\}$. 

Let $R \leq X$ and $K \ll RN^\epsilon \ll RT^{ \epsilon}$. It is sufficient to consider the dyadic sum
\es{ \label{def:IRK} \mathcal I(R, K; \ell) &= T \sum_{r \sim R} \frac{1}{r^2} \sum_{|k| \sim K} \int_{-T^\epsilon}^{T^\epsilon} g(u)  \left| \frac{1}{\sqrt N}\sum_n A(1, \ell, n)S(k, n; r) w_3\left(\frac nN\right) e\pr{\frac{un}{rT}} \right|^2 \> du.
}
where 
\es{\label{def:gu} g(u) = g(u, r, k) = \min \left\{\frac{1}{|u|}, \frac{R/K}{1+u^2}\right\}.}

It now suffices to prove the following Lemma.
\begin{lem} \label{lem:mainlem2} For any fixed $\ell \ll T^{1 + \epsilon}$, let $T \ll N \ll \frac{T^{2 + \epsilon }}{\ell^2}$, $R \leq X$, $K \ll RT^{\epsilon}$, and $X = \min \left\{T, \frac NT \right\}.$ Then 
	\est{\mathcal I(R, K; \ell) \ll T^{2 + \epsilon}}
	where the implied constant depends on $\epsilon.$ This bound is uniform in $\ell$.
	 
\end{lem}

The proof of Lemma \ref{lem:mainlem2} will be provided Section \ref{sec:IRKinitial} - \ref{sec:big}.  To orient the reader, we provide an outline of this important Lemma.

\subsection{Outline of the proof of Lemma \ref{lem:mainlem2}} \label{subsec:sketch}
For this outline, we will ignore technical issues and focus on structural features of the proof.  First, we consider the sum
\begin{equation}\label{eqn:IRK1}
\mathcal I(R, K; \ell, U ) := \sum_{r \sim R} \frac{1}{U r^2} \sum_{|k| \sim K} \int_{U}^{2U} \left| \frac{1}{\sqrt N}\sum_n A(1, \ell, n)S(k, n; r) w_3\left(\frac nN\right) e\pr{\frac{un}{rT}} \right|^2 \> du,
\end{equation}
where we have applied a dyadic subdivision to the integral over $u$ and replaced $g(u)$ by $\frac{1}{U}$.  For illustrative purposes, we will assume that $T^{-100} \le U \le T^\epsilon$.    Applying Poisson over $k$ and Voronoi for the sum over $n$ along with an application of Cauchy-Schwarz results in sums roughly of the form 
\begin{equation}\label{eqn:IRK2}
\frac{T^{1 + \epsilon}R}{U N}     \int_{U}^{2U} \sum_{r \sim R} \sumstar_{\substack{a_1 \mod r }} \left|  \ \sum_{m > 0} \frac{A(m, 1, 1)}{m} \mathcal {KL}( \overline{a_1},\, m \, ; r, (1, 1) , (1, 1)) F\pr{\frac{m }{r^4}}\right|^2 \> du,
\end{equation}
where $ \mathcal {KL}(\overline{a_1},\, m \, ; r, (1, 1) , (1, 1))$ denotes the usual hyper-Kloosterman sum to be defined later and $F$ is some type of integral transform of $w_3\bfrac{n}{N} e\bfrac{un}{rT}$.  In the $GL(3)$ case in Young's work \cite{Young}, Equation \eqref{eqn:IRK2} is already sufficient, since the sum over $m$ will be non-existent.  In the $GL(4)$ case, the dual sum over $m$ can be essentially the same length as the original sum over $n$ which presents significant difficulties.  Here, it is useful to note that the extreme case $N = T^2$, $R = T$ is not the most difficult, and indeed follows by an application of the large sieve.  

To be more precise, by completing the sum over $a_1$ to all $a_1 \bmod r$ and a similar procedure upon opening up the hyper-Kloosterman sum, we get a sum of the form
\begin{equation}\label{eqn:IRKfourier}
\frac{T^{1 + \epsilon}R^3}{UN}     \int_{U}^{2U} \sum_{r \sim R}  \sumstar_{x \mod{r}} \left| \sum_{m > 0} \frac{A(m, 1, 1) }{m} F\pr{\frac{m}{r^4}} e\pr{\frac{m\,x}{r}}
\right|^2 du.
\end{equation}
In the actual proof we then proceed to examine two separate cases: when $m$ is small and when $m$ is large.  However, as a conceptual framework, the reader should think of the following two cases instead: when the phase inside $F$ is insignificant and when the phase plays a significant role.  

The first case includes the case when $m$ is small but also includes the case when all the parameters are large, specifically when $N = T^2, R=T$ and the $m$ is of size $R^4/N = T^2$.  Ignoring all technical complications inside $F$, we will see that this case can be handled by an application of the large sieve.

The second case is more involved.  For motivation, note that in \eqref{eqn:IRK1}, one may write
$$\int \left|\sum_n e\bfrac{nu}{rT}...\right|^2 = \sum_{n_1, n_2}... \int_U^{2U} e\bfrac{u(n_1-n_2)}{rT} du,
$$and the integral in $u$ essentially forces 
\begin{equation}\label{eqn:IRK3}
|n_1-n_2| \ll \frac{RT^{1+\epsilon}}{U}
\end{equation} which is significant when $\frac{RT}{U}$ is smaller than $N$.  (Note that this essentially excludes the extreme case $N = T^2$ and $R = T$.)  We thus assume that
$$RT \ll NU.
$$

The presence of this narrow region type condition on $(n_1, n_2)$ should be no surprise, since our initial sum looked like $\sum_{t_j} ...\bfrac{n_1}{n_2}^{it_j}$ where $1/4+t_j^2$ are Laplace eigenvalues, and we expect this average to force $n_1$ and $n_2$ to be close.  (This expectation is obscured after applying Kutznetsov and other tools, but \eqref{eqn:IRK3} is a direct descendant.)  We should expect to understand this natural constraint well if we are to derive cutting edge results.

Morally, the main issue is how the condition \eqref{eqn:IRK3} is expressed in \eqref{eqn:IRKfourier} after Voronoi summation transform the sum over $n$ into dual sums over $m$.  Of course, when $n_1$ and $n_2$ corresponds to variables of integration $y_1$ and $y_2$ inside $F$, this forces $y_1$ and $y_2$ to be close.  Writing the dual sums as
$$\left|\sum_m ...\right|^2 = \sum_{m_1, m_2}...,
$$a less obvious conclusion is that this in turn forces $m_1$ and $m_2$ to be somewhat close, which arises from the phases introduced by Voronoi on $GL(4)$.  Actually, this not surprising in hindsight,  In particular, one can morally express the condition $|n_1-n_2| \ll RT^{1+\epsilon}$ via an integration
$$\frac{1}{V} \int_V^{2V} \bfrac{n_1}{n_2}^{it} dt
$$where $V \asymp \frac{NU}{RT}$.  The combined conductor of $n^{it} e\bfrac{mx}{r}$ is then $\asymp \frac{NU}{T}$ is independent of $R$.  

At this point, one can try to apply the hybrid large sieve, without applying Voronoi in $n_i$.  This will give a bound roughly of the form $R^2 + \frac{T^2}{U}$, which is acceptable only when $U \gg T^{-\epsilon}$ say.  In other words, we lose in the case when $n_1$ and $n_2$ are not forced to be very close.  However, recall that the hybrid conductor mentioned is around size $\frac{NU}{T}$ and is small when $U$ is small, and this explains when we should expect Voronoi to give us savings.

Although this specific formula does not appear to be readily available in the literature, we would still expect some formula like
$$\sum_n n^{it} e\bfrac{xn}{r} g\bfrac{n}{N} = \sum_m m^{-it} \times \textup{ hyperkloosterman sum} \times \textup{integral transform} 
$$to hold.  Thus, we would expect $\bfrac{n_1}{n_2}^{it}$ to transform to $\bfrac{m_2}{m_1}^{it}$ and so the integral over $t$ forces the dual variables $m_1$ and $m_2$ to be close also.

In the actual proof, we use the Voronoi summation formula as in Theorem \ref{thm:voronoi}.  For this reason and due to other suppressed details, the actual computation is rather intricate.  After understanding the dual sum over $m_i$ as described above, we can use the hybrid large sieve to handle both the condition that $m_1$ and $m_2$ are close, and the resultant linear phases.  In our exposition, we choose instead to use a dissection to separate the interdependence of the $m_i$.  This is purely a matter of preference; for us, this dissection is more direct and easier to visualize, but slightly lengthier.  This analysis takes up the bulk of the remaining work in \S \ref{sec:big}.

\section{Voronoi summation}
In this section, we collect some technical lemmas which arise when using Voronoi summation. We will be using the Voronoi formula over $GL(4)$ from Miller and Schmid's work \cite{MS} to deal with  the sum over $n$ in $\mathcal I(R, K; \ell )$.  First, we introduce necessary notations, which are taken from \cite{MS} and \cite{MZ}.

Let $a, n \in \mathbb Z$ and $r \in \mathbb N$ and define 
$$ \mathbf{ q }= (q_1, q_2)  \ \ \ \ \ \textrm{and} \ \ \ \ \ \mathbf{d} = (d_1, d_2)$$
to be vectors of positive integers, where $d_1 | q_1r$ and $d_2 | \frac{q_1q_2r}{d_1}$. The hyper-Kloosterman sum is defined to be
\est{\label{def:hyperkl} \mathcal {KL}(a, n, r; \mathbf{q}, \mathbf{d}) = \sumstar_{x_1 \mod{\frac{q_1r}{d_1}}} \, \sumstar_{x_2 \mod{\frac{q_1q_2r}{d_1d_2}}}  \e{\frac{d_1x_1a}{r} + \frac{d_2x_2 \overline{x}_1}{\frac{q_1r}{d_1}} + \frac{n \overline{x}_2}{\frac{q_1q_2r}{d_1d_2}} },}
where we recall that $\sumstar$ is defined as in Theorem \ref{thm:young7}. 

Next let $\psi $ be a smooth function on $\mathbb R$ and compactly supported in $(0, \infty)$ and away from 0, and define the integral transform of $\psi$ to be 
\est{ \Psi(y) = \int_{\mathbb R^4} \psi \left( \frac{x_1x_2x_3x_4}{y}\right) \prod_{j = 1}^4 \pr{\e{ -x_j} |x_j|^{-\lambda_j} \sgn(x_j)^{\delta_j}  \> dx_j},}
where $\vec \lambda = (\lambda_1, ..., \lambda_4)$, $\vec \delta = (\delta_1, ..., \delta_4)$, and $(\vec \lambda, \vec \delta) \in \mathbb C^n \times (\mathbb Z/2\mathbb Z)^n$ is the representation parameter of a cusp form on $GL(4)$.

 The function $\Psi$ can be reformulated through Mellin transform. This was done in \cite{MS} and \cite{MZ}, and we will quote the results here. 

Let $\widetilde{\psi}(s)$ is the usual Mellin transform of $\psi$ defined by
\begin{equation}
\widetilde{\psi}(s) = \int_0^\infty \psi(t) \frac{t^s}{t} dt.
\end{equation}The integral above converges for all $s \in \mathbb{C}$ since $\psi$ is compactly supported away from $0$. Next we define
$$ \fG_{\delta}(s) := \left\{ \begin{array}{ll}
\frac{\Gamma_{\mathbb R} (s)}{\Gamma_{\mathbb R}( 1- s)} &  \ \textrm{if} \  \delta \in 2 \mathbb Z \\
i \frac{\Gamma_{\mathbb R} (s+1)}{\Gamma_{\mathbb R}( 2- s)}& \ \textrm{if} \ \delta \in 2 \mathbb Z + 1,
\end{array}  \right. $$
where $\Gamma_{\mathbb R}(s) = \pi^{-s/2} \Gamma(s/2)$. Now, define

$$ G_+(s) = \prod_{j = 1}^4 \fG_{\delta_j}(s + \lambda_j)^{-1},  \ \ \ \ \ \ \ G_-(s) = \prod_{j = 1}^4 \fG_{1 + \delta_j} (s + \lambda_j)^{-1},$$
and for $\sigma > 0$ let 
\es{\label{def:Psipm} \Psi_\pm (x) = \frac{1}{2\pi i} \int_{(-\sigma)} \widetilde \psi(s) x^{s} G_{\pm}(s) \> ds .}

Then when $x > 0$ we can write
\es{\label{reformulatePsi} \Psi\pr{ x} = \Psi_+(x) + \Psi_-(x) \ \ \ \ \ \  \textrm{and} \ \ \ \ \ \ \Psi(-x) = \Psi_+(x) - \Psi_-(x).}

Now we are ready to state the Voronoi formula from \cite{MS}.
\begin{thm} \label{thm:voronoi} Let $a \in \mathbb Z$, $r \in \mathbb N$, $(a, r) = 1$, and $\psi $ be a smooth function on $\mathbb R$ and compactly supported in $(0, \infty)$ and away from 0. With notations as above, 
\est{	\sum_{n \neq 0} &A(q_2, q_1, n) \e{\frac{an}{r}} \psi(n) \\
& = r \sum_{d_1 |rq_1} \sum_{d_2 | \frac{rq_1q_2}{d_1}} \sum_{m \neq 0} \frac{A(m, d_2, d_1)}{|m|d_1d_2} \mathcal {KL}(\overline{a}, n; r, \mathbf{q}, \mathbf{d} ) \Psi\left(\frac{md_2^2d_1^3}{r^4 q_2q_1^2}\right).}
\end{thm}


Since $\Psi$ can be written as $\Psi_{\pm}$ in Equation (\ref{reformulatePsi}), we focus only on studying $\Psi_{\pm}$. Now we will find an asymptotic formula for $\Psi_{\pm}$, and this will help in analysing saddle points and main terms. From now on, we will fix $\lambda_j = \alpha_j$ as in Equation (\ref{def:alphai}) and $\delta_j = 0$ and prove the following Lemma.

\begin{lem}  \label{lem:asymforpsi} Let $N \geq 1$ ,  $\lambda_j = \alpha_j$ as in Equation (\ref{def:alphai}) and $\delta_j = 0$. Suppose $\psi(x)$ is a smooth function compactly supported on $[N, 2N]$, and $\Psi_\pm(x)$ be defined as in Equation (\ref{def:Psipm}). Then for any fixed positive integer $\mathcal K \geq 1$ and  $xN \gg 1$,
	\es{\label{psiboundforbigx} \Psi_+(x) &= x \int_0^\infty \psi(y) \sum_{j = 1}^{\mathcal K} \frac{1}{(xy)^{\frac j4 + \frac 18}} \left[ c_j \e{4(xy)^{\frac 14}} \mathcal W_j(8\pi (xy)^{\frac 14}) + d_j \e{ - 4(xy)^{\frac 14}} \overline{\mathcal W_j}(8\pi (xy)^{\frac 14}) \right]  \> dy \\
		& \hskip 4in + O\pr{\pr{xN}^{\frac{-K+3}{4} - \frac 18}} }
	where function $\mathcal W_j$ is a finite linear combination of $W_k$ defined in Lemma \ref{lem:Besselresult} and $c_j, d_j$ are suitable constants depending on $\alpha_i.$ Moreover, $\Psi_-(x)$ has the same expression except value of constants. 
	
\end{lem}
We will focus only on proving $\Psi_+$ as the proof of $\Psi_-$ can be proceeded in the same way. The proof will follow the idea of the proof of Lemma 3 in \cite{Ivic} and Lemma 6.1 in \cite{Li}, which will be provided  in Appendix \ref{sec:proofofasympPsi}.



\section{Fourier analysis on $\mathcal I(R, K; \ell)$} \label{sec:IRKinitial}
Firstly, we square out the expression of $\mathcal I(R, K; \ell)$ in (\ref{def:IRK}), put a smooth weight in $k$, and use the fact that $\frac {1}{r^2} \ll \frac{1}{R^2}$. We then obtain that 
\est{ \mathcal I(R, K; \ell) &\ll \frac{T}{NR^2} \sum_{r \sim R}  \sum_{k}  w_1\left( \frac kK \right) \int_{-T^{\epsilon}}^{T^{\epsilon}} g(u)\sumtwo_{n_1, n_2} A(1, \ell, n_1) A(1, \ell, n_2) w_3\left(\frac {n_1}N\right) w_3\left(\frac {n_2}N\right) \\
	 &\times \sumstar_{a_1 \mod r} \sumstar_{a_2 \mod r} \e{\frac{a_1k - a_2k}{r}} \e{\frac{\bar{a}_1n_1 - \bar{a}_2n_2}{r}} \e{\frac{un_1 - un_2}{rT}}  \> du,}
 where we recall that $w_1$ is a smooth compactly supported function defined as above.  Then we apply Poisson summation to the sum over $k$.
 \est{\sum_{k}& w_1\pr{\frac kK}\e{\frac{(a_1 - a_2)k}{r}} = \sum_{c \mod r} \e{\frac{(a_1 - a_2)c}{r}} \sum_{k \equiv c \mod r} w_1\pr{\frac kK} \\
&= \sum_{c \mod r} \e{\frac{(a_1 - a_2)c}{r}} \sum_{j } \frac{K}{r} \e{\frac{cj}{r}} \hat w_1\pr{\frac{Kj}{r}} = K \sum_{j \equiv a_2 - a_1 \mod r} \hat w_1\pr{\frac{Kj}{r}}.}
Therefore
\es{\label{boundI1} \mathcal I(R, K ; \ell) &\ll  \frac{TK}{NR^2} \sum_{r \sim R}  \sum_{j}  \hat w_1\left( \frac{Kj}{r} \right)  \sumstar_{\substack{a_1 \mod r \\ (a_1 + j, \ r) = 1}} \int_{-T^{\epsilon}}^{T^{\epsilon}}g(u) \sumtwo_{n_1, n_2} A(1, \ell, n_1) A(1, \ell, n_2) w_3\left(\frac {n_1}N\right) w_3\left(\frac {n_2}N\right) \\
	&\times   \e{\frac{\bar{a}_1n_1 - \overline{(a_1 + j)}n_2}{r}} \e{\frac{un_1 - un_2}{rT}}  \> du \\
&\ll  \frac{TK}{NR^2} \sum_{r \sim R}  \sum_{j \ll \frac{RT^{\epsilon}}{K}}  \sumstar_{\substack{a_1 \mod r \\ (a_1 + j, \ r) = 1}} \int_{-T^{\epsilon}}^{T^{\epsilon}} g(u)\left( \left|  \sum_{n}A(1, \ell, n) w_3\left(\frac {n}N\right)  \e{\frac{\bar{a}_1n}{r}}\e{\frac{un}{rT}}\right|^2 \right. \\
& \ \ \ \ \ + \left. \left|  \sum_{n}A(1, \ell, n) w_3\left(\frac {n}N\right)  \e{\frac{\overline{(a_1 +j)}n}{r}}\e{\frac{un}{rT}}\right|^2  \right) \>du \\
&\ll  \frac{T^{1+\epsilon}}{NR}  \int_{-T^{\epsilon}}^{T^{\epsilon}} g(u)  \sum_{r \sim R} \sumstar_{\substack{a_1 \mod r }} \left|  \sum_{n}A(1, \ell, n) w_3\left(\frac {n}N\right) \e{\frac{un}{rT}} \e{\frac{{a}_1n}{r}}\right|^2 \> du, }
where we note that the second term involving $\e{\frac{\overline{(a_1+j)}n}{r}}$ becomes $\e{\frac{\bar a_1 n}{r}}$ by a change of variables on $a_1$, and we have extended the sum over $a_1$ by positivity.

Note that $N \gg TR$ since $R \ll X \le \frac{N}{T}$.  We apply the Voronoi Summation formula (Theorem \ref{thm:voronoi}) to the sum over $n$  in Equation (\ref{boundI1}). 
Let 

\es{\label{def:fx} f(x; u, r) := f(x) =  w_3\left(\frac {x}N\right) \e{\frac{ux}{rT}}. }

Then

\es{\label{eqn:aftervoronoi} &\sum_{n}A(1, \ell, n) f(n) \e{-\frac{{a}_1n}{r}} \\
	&= |r| \sum_{d_1 | r \ell }  \ \sum_{d_2 | \frac{r\ell}{d_1}} \ \sum_{m \neq 0} \frac{A(m, d_2, d_1)}{|m |d_1d_2} \mathcal{KL }( \overline{a_1},\, m \, ; r, (\ell, 1) , (d_1, d_2)) F\pr{\frac{m \, d_2^2 \, d_1^3}{r^4\,\ell^2}; r} ,}
where $F$ is defined analogously to $\Psi$, and 
\est{\mathcal {KL}(\overline{a_1}, \,m \, ; r, ( \ell, 1 ), (d_1, d_2)) =  \sumstar_{x_1 \mod {\frac{r\ell}{d_1}}} \sumstar_{x_2 \mod {\frac{r\ell}{d_1d_2}} } \e{\frac{d_1x_1\overline{a_1}}{r} + \frac{d_2x_2\overline{x_1}}{\frac{ r\ell}{d_1}}  + \frac{m\,\overline{x_2}}{\frac{ r\ell}{d_1d_2}} } .}

After Cauchy-Schwarz inequality in $d_1, d_2$ and considering only positive $m$ due to symmetry, now we need to bound
\est{&\mathcal{I}_1 (R, K; \ell) := \frac{T^{1 + \epsilon}R}{N}     \int_{-T^{\epsilon}}^{T^{\epsilon}} g(u) \sum_{r \sim R} \ \sumstar_{\substack{a_1 \mod r }} \sum_{d_1 | r\ell} \sum_{d_2 | \frac{r\ell}{d_1}} \frac{1}{d_1^2d_2^2}\\
	&\hskip 1.5in \times \left|  \ \sum_{m > 0} \frac{A(m, d_2, d_1)}{m} \mathcal {KL}( \overline{a_1},\, m \, ; r, (\ell, 1) , (d_1, d_2)) F_{\pm}\pr{\frac{md_2^2 d_1^3 }{r^4\,\ell^2}; r}\right|^2 \> du.}

\section{Simplifying exponential sums} \label{sec:simplifyIRK}

In this section, we deal with the exponential sum in the hyper-Kloosterman sum. Moreover, the bound for $F_-$ can be evaluated in the same way as $F_+$, so we consider only $F_+$. By Cauchy-Schwarz inequality, changing variable from $a_1$ to $\overline{a_1}$ and completing summation over $a_1$, we have $\mathcal I_1(R,K ; \ell)$ is bounded by
\est{ &\ll \frac{T^{1 + \epsilon}R}{N}     \int_{-T^{\epsilon}}^{T^{\epsilon}} g(u) \sum_{r \sim R} \sum_{\substack{a_1 \mod r }}  \sum_{d_1 | r\ell} \sum_{d_2 | \frac{r\ell}{d_1}} \frac{1}{d_1^2d_2^2}\\
	&\hskip 1.5in \times \left|  \ \sum_{m > 0} \frac{A(m, d_2, d_1)}{m} \mathcal {KL}( \overline{a_1},\, m \, ; r, (\ell, 1) , (d_1, d_2)) F_{+}\pr{\frac{md_2^2 d_1^3 }{r^4\,\ell^2}; r}\right|^2 \> du. \\
&= \frac{T^{1 + \epsilon}R}{N}     \int_{-T^{\epsilon}}^{T^{\epsilon}} g(u) \sum_{r \sim R} \sum_{\substack{a_1 \mod r }} \sum_{d_1 | r\ell} \sum_{d_2 | \frac{r\ell}{d_1}} \frac{1}{d_1^2d_2^2} \\
& \ \ \ \times  \sum_{m_1, m_2 > 0} \frac{A(m_1, d_2, d_1) \overline{A(m_2, d_2, d_1)}}{m_1m_2}  F_{+}\pr{\frac{m_1d_2^2d_1^3}{r^4\,\ell^2}; r} \overline{F_{+}\pr{\frac{m_2d_2^2d_1^3 }{r^4\,\ell^2}; r} } \sumstar_{x_1 \mod{\frac{r\ell}{d_1}}} \sumstar_{x_1' \mod{\frac{r\ell}{d_1}}}  \\
&\ \ \ \times \sumstar_{x_2 \mod {\frac{r\ell}{d_1d_2}} } \sumstar_{x_2' \mod {\frac{r\ell}{d_1d_2}} } \e{\frac{d_1(x_1 - x_1'){a_1}}{r} + \frac{d_2(x_2\overline{x_1} - x_2'\overline{x_1'}) }{\frac{r\ell}{d_1}}  + \frac{m_1\,\overline{x_2} - m_2\,\overline{x_2'}}{ \frac{r\ell}{d_1d_2}} } \> du. }
Next we sum over $a_1$ and see that $d_1x_1 \equiv d_1x_1' \bmod r$ by orthogonality, which impllies $x_1 \equiv x_1' \bmod \frac{r}{(r, d_1)}$.  Thus we may write $x_1' = x_1 + \frac{r}{(r, d_1)}y$, where $y$ runs through those residues mod $\frac{(r, d_1)\ell}{d_1}$ such that $(x_1 + \frac{r}{(r, d_1)}y, \frac{r\ell}{d_1}) = 1$.  For simplicity, let $\sumsharp_{y \bmod \frac{(r, d_1)\ell}{d_1}}$ denote the sum over such $y$.  Thus our sum becomes
\es{ \label{eqn:sum1} &\frac{T^{1 + \epsilon}R}{N}     \int_{-T^{\epsilon}}^{T^{\epsilon}} g(u) \sum_{r \sim R} r \sum_{d_1 | r\ell} \sum_{d_2 | \frac{r\ell}{d_1}} \frac{1}{d_1^2d_2^2} \sumstar_{x_1 \bmod \frac{r \ell}{d_1}} \;\; \sumsharp_{y \bmod \frac{(r, d_1)\ell}{d_1}}  S_1 S_2\> du \\ 
	&\le \frac{T^{1 + \epsilon}R^2}{N}     \int_{-T^{\epsilon}}^{T^{\epsilon}} g(u) \sum_{r \sim R} \sum_{d_1 | r\ell} \sum_{d_2 | \frac{r\ell}{d_1}} \frac{1}{d_1^2d_2^2}  \sumstar_{x_1 \bmod \frac{r \ell}{d_1}} \;\; \sumsharp_{y \bmod \frac{(r, d_1)\ell}{d_1}}  (|S_1|^2 +  |S_2|^2 )\> du,
}where
\begin{align*}
S_1 = \sum_{m_1 > 0} \frac{A(m_1, d_2, d_1) }{m_1}  F_{+}\pr{\frac{m_1d_2^2d_1^3}{r^4\,\ell^2}; r}  \sumstar_{x_2 \bmod {\frac{r\ell}{d_1d_2}} } \e{ \frac{d_2x_2\overline{x_1}}{\frac{r\ell}{d_1}}  + \frac{m_1\,\overline{x_2}}{\frac{r\ell}{d_1d_2}} },
\end{align*}and
\begin{align*}
S_2  = \sum_{m_2 > 0} \frac{\overline{A(m_2, d_2, d_1)}}{m_2}   \overline{F_{+}\pr{\frac{m_2d_2^2d_1^3 }{r^4\,\ell^2}; r}}\sumstar_{x_2' \bmod {\frac{r\ell}{d_1d_2}} } \e{ \frac{-d_2x_2'\overline{x_1 + \frac{r}{(r, d_1)}y} }{\frac{r\ell}{d_1}}  + \frac{- m_2\,\overline{x_2'}}{\frac{r\ell}{d_1d_2}} }.
\end{align*}
Inside $S_2$, we may use the change of variables $u = \overline{x_1 + \frac{r}{(r, d_1)}y}$.  The condition on $y$ then becomes that $(\overline{u} - \frac{r}{(r, d_1)}y, \frac{r\ell}{d_1}) = 1$.  After this change of variables, we extend the sum over $y$ to all residues mod $\frac{(r, d_1)\ell}{d_1}$.  Thus,
\begin{align*}
&\sum_{r \sim R}  \sum_{d_1 | r\ell} \sum_{d_2 | \frac{r\ell}{d_1}} \frac{1}{d_1^2d_2^2}  \sumstar_{x_1 \bmod{\frac{r\ell}{d_1}}} \;\;\sumsharp_{y \bmod \frac{(r, d_1)\ell}{d_1}} |S_2|^2 \\
&\le \sum_{r \sim R}  \sum_{d_1 | r\ell} \sum_{d_2 | \frac{r\ell}{d_1}} \frac{1}{d_1^2d_2^2}   \sumstar_{u \bmod{\frac{r\ell}{d_1}}}\;\;\sum_{y \bmod \frac{(r, d_1)\ell}{d_1}} |S_2|^2 = \frac{(r, d_1)\ell}{d_1} \sum_{r \sim R} \sum_{d_1 | r\ell} \sum_{d_2 | \frac{r\ell}{d_1}} \frac{1}{d_1^2d_2^2} \sumstar_{x_1 \bmod{\frac{r\ell}{d_1}}} |S_1|^2.
\end{align*}By a further change of variables from $x_1 $ to $\overline{x_1}$, the fact that $S_1$ is independent of $y$ and $\frac{(r, d_1)}{d_1} \leq 1$, the quantity in \eqref{eqn:sum1} is bounded by
\begin{align*}
&\ell \frac{T^{1 + \epsilon}R^2}{N}     \int_{-T^{\epsilon}}^{T^{\epsilon}}g(u) \sum_{r \sim R}  \sum_{d_1 | r\ell} \sum_{d_2 | \frac{r\ell}{d_1}} \frac{1}{d_1^2d_2^2} \sumstar_{x_1 \bmod{\frac{r\ell}{d_1}}}|S_1|^2 du \\
&\leq \ell \frac{T^{1 + \epsilon}R^2}{N}     \int_{-T^{\epsilon}}^{T^{\epsilon}}g(u) \sum_{r \sim R} \sum_{d_1 | r\ell} \sum_{d_2 | \frac{r\ell}{d_1}} \frac{1}{d_1^2d_2^2}  \\
&\times \sumstar_{x_1 \bmod{\frac{r\ell}{d_1}}}\left| \sum_{m > 0} \frac{A(m, d_2, d_1) }{m_1}  F_{+}\pr{\frac{md_2^2d_1^3}{r^4\,\ell^2}; r}  \sumstar_{x_2 \bmod {\frac{r\ell}{d_1d_2}} } \e{ \frac{d_2x_2 x_1}{\frac{r\ell}{d_1}}  + \frac{m\,\overline{x_2}}{\frac{r\ell}{d_1d_2}} }
\right|^2 du.
\end{align*}
Now we may extend the sum over $x_1$ to all residues mod $\frac{r\ell}{d_1}$ by positivity.  Opening the square produces two sums $x_2, x_2' \bmod \frac{r\ell}{d_1d_2}$.  However, by orthogonality, the sum over $x_1$ gives the condition $d_2x_2 \equiv d_2x_2' \bmod \frac{r\ell}{d_1}$, which implies $x_2 \equiv x_2' \bmod \frac{r\ell}{d_1d_2}$ because $d_2 | \frac{r\ell}{d_1}$. So the above sum is
\begin{align*} 
\ll \ell^2 \frac{T^{1 + \epsilon}R^3}{N}     \int_{-T^{\epsilon}}^{T^{\epsilon}} g(u)\sum_{r \sim R}  \sum_{d_1 | r\ell} \sum_{d_2 | \frac{r\ell}{d_1}} \frac{1}{d_1^3d_2^2} \sumstar_{x \bmod{\frac{r\ell}{d_1d_2}}} \left| \sum_{m > 0} \frac{A(m, d_2, d_1) }{m}  F_{+}\pr{\frac{md_2^2d_1^3}{r^4\,\ell^2}; r} \e{\frac{m\,x}{\frac{r\ell}{d_1d_2}}}
\right|^2 du,
\end{align*} 
where we have used a change of variables $x = \overline{x_2}$. Next we write $r_1 = r\ell$, switch the sums $d_1, d_2$ and $r$ and drop condition $\ell | r_1$. Thus the above expression is 
 \es{ \label{eqn:sum3}
&\ll \ell^2 \frac{T^{1 + \epsilon}R^3}{N}     \int_{-T^{\epsilon}}^{T^{\epsilon}} g(u)\sum_{d_1 \ll R\ell} \sum_{d_2 \ll R\ell} \frac{1}{d_1^3d_2^2} \sum_{\substack{r_1 \sim R\ell \\ d_1d_2 | r_1}}   \sumstar_{x \bmod{\frac{r_1}{d_1d_2}}} \\
&\hskip 1.5in\left| \sum_{m > 0} \frac{A(m, d_2, d_1) }{m}  F_{+}\pr{\frac{md_2^2d_1^3 \ell^2}{r_1^4}; \frac{r_1}{\ell}} \e{\frac{m\,x}{\frac{r_1}{d_1d_2}}}
\right|^2 du \\
&= \ell^2 \frac{T^{1 + \epsilon}R^3}{N}     \int_{-T^{\epsilon}}^{T^{\epsilon}} g(u)\sum_{d_1 \ll R\ell} \sum_{d_2 \ll R\ell} \frac{1}{d_1^3d_2^2} \sum_{\substack{r_1 \sim \frac{R\ell}{d_1d_2} }}   \sumstar_{x \bmod{r_1}} \\
&\hskip 1.5in\left| \sum_{m > 0} \frac{A(m, d_2, d_1) }{m}  F_{+}\pr{\frac{m\ell^2}{r_1^4d_1d_2^2}; \frac{r_1d_1d_2}{\ell}} \e{\frac{m\,x}{r_1}}
\right|^2 du
} 

Now we split $m$ into two ranges.  We let $\mathcal I_{sm}(R, K; \ell)$ be the expression on the right side of \eqref{eqn:sum3} with $m \le \frac{R^4\ell^2}{Nd_2^2d_1^3} T^{\epsilon_1}$ and $\mathcal I_{big}(R, K; \ell )$ be the same expression for $m > \frac{R^4\ell^2}{Nd_2^2 d_1^3} T^{\epsilon_1}$, where $\epsilon_1$ is a fixed small constant to be determined later. 

Since $|a+b|^2  \le 2(|a|^2 + |b|^2)$, it now suffices to prove the following Propositions.

\begin{prop}\label{prop:sm} With notations defined as above, 
	\begin{equation*}
	\mathcal I_{sm}(R, K ; \ell) \ll T^{2 + \epsilon}
	\end{equation*}
where the implied constant depends on $\epsilon$.
\end{prop}

\begin{prop}\label{prop:big} With notations defined as above,
	\begin{equation*}
	\mathcal I_{big}(R, K; \ell) \ll T^{2 + \epsilon}
		\end{equation*}
	where the implied constant depends on $\epsilon$.
\end{prop}

We prove Proposition \ref{prop:sm} in Section \ref{sec:sm} and Proposition \ref{prop:big} in Section \ref{sec:big}. Moreover, we note that the presence of $d_1$ and $d_2$ are conceptually unimportant and gives rise to convoluted notation which obfuscates the main ideas.  To ease the notational burden, readers may set $d_1 = d_2 = 1$ in Section \ref{sec:sm} and \ref{sec:big} below.  
 
\section{Proof of Proposition \ref{prop:sm}}\label{sec:sm}

We would like to apply the usual Large Sieve to Equation \ref{eqn:sum3}, but we first need to make the inner sum independent of $r$ and $\ell$.  To do this, we first note that by the work of Luo, Rudnick and Sarnak \cite{LRS},
\begin{equation} \label{eqn:alphajbdd}|\Re \alpha_j|  < \frac{1}{2}. 
\end{equation}

We express $F_+$ as in (\ref{def:Psipm}), use the change of variable $\frac{1-s}{2} \rightarrow s $, and derive that
\est{\label{eqn:F+2}F_+(x; \frac{r_1d_1d_2}{\ell}) 
	&= \frac{2x \pi^2}{2\pi i} \int_{(\sigma_1)}  \frac{\pi^{-8s}x^{-2s} \Gamma\pr{s- \frac{ \alpha_1}{2}}\Gamma\pr{s- \frac{ \alpha_2}{2}}\Gamma\pr{s - \frac{ \alpha_3}{2}}\Gamma\pr{s - \frac{\alpha_4}{2}}}{\Gamma\pr{\frac 12 -s +\frac{ \alpha_1}{2}}\Gamma\pr{\frac 12 - s + \frac{  \alpha_2}{2}}\Gamma\pr{\frac 12 - s +\frac{ \alpha_3}{2}}\Gamma\pr{\frac 12 - s + \frac{ \alpha_4}{2}}}\widetilde{f}(-2s + 1)  \> ds,}
so we may shift the contour of integration to $\sigma_1 = 1/4$ without crossing any poles of the integrand by \eqref{eqn:alphajbdd}.  We proceed to further shift the contour to $\sigma_1 < 1/8$.  We may pick up residues of the form
\begin{equation}\label{eqn:res}
C x^{1-2s_0} \widetilde{f}(1-2s_0),
\end{equation}
where $\tRe s_0 < 1/4$ and for some constant $C$ only dependent on $\alpha_j$.

Further note that
\begin{equation}\label{eqn:ftildebdd}
\widetilde{f}(-2s + 1) = \int_0^\infty w_3\bfrac{y}{N} e\bfrac{u\ell y}{r_1d_1d_2T} y^{-2s} dy \ll N^{1-2s}.
\end{equation}
For simplicity, we write
$$\cG(s) =  \frac{\Gamma\pr{s- \frac{ \alpha_1}{2}}\Gamma\pr{s- \frac{ \alpha_2}{2}}\Gamma\pr{s - \frac{ \alpha_3}{2}}\Gamma\pr{s - \frac{\alpha_4}{2}}}{\Gamma\pr{\frac 12 -s +\frac{ \alpha_1}{2}}\Gamma\pr{\frac 12 - s + \frac{  \alpha_2}{2}}\Gamma\pr{\frac 12 - s +\frac{ \alpha_3}{2}}\Gamma\pr{\frac 12 - s + \frac{ \alpha_4}{2}}} \widetilde{f}(-2s + 1) 
$$and note that
\begin{equation}\label{eqn:Gbdd}
\cG(s) \ll \frac{N^{1-2\tRe s}}{|s|^{1+\eps}}
\end{equation}
when $\tRe s < 1/8$ by Stirling's formula and \eqref{eqn:ftildebdd}.  This is the main motivation behind shifting the contour to $\tRe s <1/8$.  Note that repeated integration by parts on $\widetilde{f}(1-2s)$ will also given sufficiently rapid decay in $s$ but at the cost of introducing factors like $\frac{Nu}{rT}$, which can be quite large.  This leads us to the examination of the contribution of the residues and the contribution of the remaining contour integral separately.

In both cases, we apply a dyadic subdivision to the sum over $m$, so that we examine sums $m \sim M$ for $M \le \frac{R^4\ell^2}{Nd_2^2 d_1^3} T^{\epsilon_1}$.  
\subsection{Contribution of residues} \label{sec:contres}
For brevity, write $\lambda = 1-2s_0$, and note that $\tRe \lambda \ge 1/2$.  We need to bound
\es{ \label{sec8.1maineq}
&\ell^2 \frac{T^{1 + \epsilon}R^3}{N}     \int_{-T^{\epsilon}}^{T^{\epsilon}} g(u)\sum_{d_1 \ll R\ell} \sum_{d_2 \ll R\ell} \frac{1}{d_1^3d_2^2} \sum_{\substack{r_1 \sim \frac{R\ell}{d_1d_2} }}   \sumstar_{x \bmod{r_1}} \\
& \hskip 2in \left|\sum_{m \sim M} \frac{A(m, d_2, d_1) }{m} \pr{\frac{m\ell^2}{r_1^4\,d_1d_2^2}}^\lambda \widetilde f(1-2s_0)\e{\frac{m\,x}{r_1}}\right|^2  du\\
&\ll \ell^2 \frac{T^{1 + \epsilon}R^3}{N}   \sum_{d_1 \ll R\ell} \sum_{d_2 \ll R\ell} \frac{1}{d_1^3d_2^2} \bfrac{NMd_2^2d_1^3}{R^4 \ell^2}^{ 2\tRe \lambda} \frac{1}{M^{2\tRe \lambda}} \sum_{\substack{r_1 \sim \frac{R\ell}{d_1d_2} }}   \sumstar_{x \bmod{r_1}}  \left|\sum_{m \sim M} \frac{A(m, d_2, d_1 ) }{m^{1-\lambda}}  \e{\frac{m\,x}{r_1}}\right|^2,
}
using the bound $\widetilde f(1-2s_0) \ll N^{1-2\tRe s} = N^{\tRe \lambda}$ and the fact that
$$ \int_{-T^{\epsilon}}^{T^{\epsilon}} g(u) du \ll \log T.
$$
Using that $M \le \frac{R^4 \ell^2}{Nd_2^2 d_1^3} T^{\epsilon_1}$ and $\tRe \lambda \ge 1/2$, we see that
$$\bfrac{NM}{R^4 \ell^2 d_2^2 d_1^3}^{ 2\tRe \lambda} \ll \bfrac{NM}{R^4 \ell^2 d_2^2 d_1^3} T^\epsilon,
$$
upon choosing  $\epsilon_1$ to be sufficiently small. Substituting this in, we see that the quantity above is bounded by
\est{&\ll \ell^2 \frac{T^{1 + \epsilon}R^3}{N}  \sum_{d_1 \ll R\ell} \sum_{d_2 \ll R\ell} \frac{1}{d_1^3d_2^2} \bfrac{NMd_2^2d_1^3}{R^4 \ell^2} \frac{1}{M^{2\tRe \lambda}} \sum_{\substack{r_1 \sim \frac{R\ell}{d_1d_2} }}   \sumstar_{x \bmod{r_1}}  \left|\sum_{m \sim M} \frac{A(m, d_2, d_1 ) }{m^{1-\lambda}}  \e{\frac{m\,x}{r_1}}\right|^2 \\
&\ll \frac{T^{1 + \epsilon}M^{1 - 2\tRe \lambda}}{R}  \sum_{d_1 \ll R\ell} \sum_{d_2\ll R\ell }   \sum_{\substack{r_1 \sim \frac{R\ell}{d_1d_2} }} \sumstar_{x \bmod{r_1}} \left|\sum_{m \sim M} \frac{A(m, d_2, d_1 ) }{m^{1-\lambda}}  \e{\frac{m\,x}{r_1}}\right|^2}

We now apply the usual large sieve to the sum over $r_1$ and $x$ to see that the above is
\es{ \label{eqn:bddpre2}
&\ll \frac{T^{1 + \epsilon}M^{1 - 2\tRe \lambda}}{R}  \sum_{d_1 \ll R\ell} \sum_{d_2\ll R\ell }   \left( \left(\frac{ R\ell}{d_1d_2}\right)^2 + M \right) \left( \sum_{m \sim M} \left|\frac{A(m, d_2, d_1) }{m^{1-\lambda}} \right|^2\right)  \\
&\ll \frac{T^{1 + \epsilon}}{MR}  \sum_{d_1 \ll R\ell} \sum_{d_2\ll R\ell }  \left( \left(\frac{ R\ell}{d_1d_2}\right)^2 + M \right) Md_1d_2
}
upon using Lemma \ref{lem:ramanujanonaverage}. Since $R \le \frac{N}{T}$, $N\ll \frac{T^{2+\epsilon}}{\ell^2}$, and  $M \ll \frac{R^4 \ell^2}{Nd_2^2 d_1^3} T^{\epsilon_1}$, choosing $\ell_1$ sufficiently small, we conclude that Equation (\ref{sec8.1maineq}) is bounded by

\es{  \label{eqn:bdd2}
&\frac{T^{1 + \epsilon}}{R} \left((\ell R)^2 + \frac{R^4\ell^2}{N} \right) \ll \ell^2 T^{1 + \epsilon}\frac{N}{T} + \frac{T^{\epsilon} N^2 \ell^2}{T^2}   \ll T^{2+\epsilon}. 
}

\subsection{Contribution of contour}

Recall now that $\sigma_1 <1$.  We get by Cauchy-Schwarz that
\begin{align*}
&\left|\int_{(\sigma_1)} \cG(s) \sum_{m \sim M} \frac{A(m, d_2, d_1) }{m} \pr{\frac{m\ell^2}{r_1^4\,d_1d_2^2}}^{1-2s} \e{\frac{m\,x}{r_1}} ds\right|^2 \\
&\ll \int_{(\sigma_1)} |\cG(s)| \left|  \sum_{m \sim M} \frac{A(m, d_2, d_1) }{m} \e{\frac{m\,x}{r_1}} \bfrac{m\ell^2}{r_1^4d_1d_2^2}^{1-2s} \right|^2ds \int_{(\sigma_1)}  |\cG(s)| ds\\
& \ll  \bfrac{N^{1/2}\ell^2}{r_1^4 d_1d_2^2}^{2\lambda_1} \int_{-\infty}^{\infty} |\cG(\sigma_1+it)|  \left|\sum_{m \sim M} \frac{A(m, d_2, d_1) }{m^{1-\lambda_1+2it}}  \e{\frac{m\,x}{r_1}}\right|^2 dt,
\end{align*}where $\lambda_1 := 1-2\sigma_1$, and since 
$$\int_{(\sigma_1)}  |\cG(s)| ds \ll N^{\lambda_1}
$$by \eqref{eqn:Gbdd}.
In order to prove Proposition \ref{prop:sm}, it suffices to prove
\es{\label{eqn:msmbdd}
 \ell^2 \frac{T^{1 + \epsilon}R^3}{N}  &   \int_{-T^{\epsilon}}^{T^{\epsilon}}g(u) \int_{-\infty}^{\infty}|\cG(\sigma_1+it)| \sum_{d_1 \ll R\ell} \sum_{d_2\ll R\ell }  \frac{1}{d_1^3d_2^2} \bfrac{NMd_2^2d_1^3}{R^4 \ell^2}^{2\lambda_1} \frac{1}{(N^{1/2}M)^{2 \lambda_1}} \\
 &\times  \sum_{\substack{r_1 \sim \frac{R\ell}{d_1d_2} }} \sumstar_{x \bmod{r_1}}   \left|\sum_{m \sim M} \frac{A(m, d_2, d_1) }{m^{1-\lambda_1 + 2it}}  \e{\frac{m\,x}{r_1}}\right|^2 dt  du \ \ \ll T^{2+\epsilon}.
}
Again since $\frac{NMd_2^2d_1^3}{R^4 \ell^2} \ll T^{\epsilon_1}$ and $\lambda_1  = 1-2\sigma_1 > 1-1/4 > 1/2$, 
$$\bfrac{NMd_2^2d_1^3}{R^4 \ell^2}^{2\lambda_1} < \frac{NMd_2^2d_1^3}{R^4 \ell^2} T^\epsilon.
$$
Thus the left hand side of the equation above is bounded by
\est{\frac{T^{1 + \epsilon} M^{1 - 2\lambda_1}}{R N^{\lambda_1}}  &   \int_{-T^{\epsilon}}^{T^{\epsilon}}g(u) \int_{-\infty}^{\infty}|\cG(\sigma_1+it)| \sum_{d_1 \ll R\ell} \sum_{d_2\ll R\ell }   \sum_{\substack{r_1 \sim \frac{R\ell}{d_1d_2} }} \sumstar_{x \bmod{r_1}}     \left|\sum_{m \sim M} \frac{A(m, d_2, d_1) }{m^{1-\lambda_1 + 2it}}  \e{\frac{m\,x}{r_1}}\right|^2 dt  du }
Similar to arguments in Section \ref{sec:contres}, we  apply the large sieve, noting that
$$\int_{-T^{\epsilon}}^{T^{\epsilon}}g(u)du \ll \log T,
$$ and obtain that the above is bounded by

\est{
	&\ll \frac{T^{1 + \epsilon}M^{1 - 2\tRe \lambda}}{R}  \sum_{d_1 \ll R\ell} \sum_{d_2\ll R\ell }   \left( \left(\frac{ R\ell}{d_1d_2}\right)^2 + M \right) \left( \sum_{m \sim M} \left|\frac{A(m, d_2, d_1) }{m^{1-\lambda}} \right|^2\right). }
Thus by the same arguments as in Equation (\ref{eqn:bddpre2}) and (\ref{eqn:bdd2}), we derive Inequality (\ref{eqn:msmbdd}) as desired.



\section{Proof of Proposition \ref{prop:big}} \label{sec:big}
The proof of Proposition \ref{prop:big} is more complex. First of all, we apply a dyadic subdivision to the sum over $m$ and the integral over $u.$ So we investigate sums $m \sim M$ for $M \geq \frac{R^4\ell^2}{Nd_2^2d_1^3}T^{\epsilon_1}$ and $u \sim U$ where $T^{-100} < |U| \ll T^{\epsilon}$.  This suffices since there are $\ll \log^2 T$ such subdivisions and since the interval $-T^{-100} \le u \le T^{-100}$ is trivially negligible.  

From Equation (\ref{eqn:sum3}), it is sufficient to consider 
\es{\label{eqn:J} \mathcal J(R, K, M, U)  &:= \int_{u \sim U} g(u)\sum_{d_1 \ll R\ell} \sum_{d_2 \ll R\ell} \frac{1}{d_1^3d_2^2} \sum_{\substack{r_1 \sim \frac{R\ell}{d_1d_2} }}   \sumstar_{x \bmod{r_1}} \\
&\hskip 1.5in\left| \sum_{m \sim M} \frac{A(m, d_2, d_1) }{m}  F_{+}\pr{\frac{m\ell^2}{r_1^4d_1d_2^2} ; \frac{r_1d_1d_2}{\ell}} \e{\frac{m\,x}{r_1}}
\right|^2 du \\
&\leq \int_{-\infty}^{\infty} g_1\pr{\frac uU} g_2(U)  \sum_{d_1 \ll R\ell} \sum_{d_2 \ll R\ell} \frac{1}{d_1^3d_2^2} \sum_{\substack{r_1 \sim \frac{R\ell}{d_1d_2} }}   \sumstar_{x \bmod{r_1}} \\
&\hskip 1.5in\left| \sum_{m \sim M} \frac{A(m, d_2, d_1) }{m}  F_{+}\pr{\frac{m\ell^2}{r_1^4d_1d_2^2}; \frac{r_1d_1d_2}{\ell}} \e{\frac{m\,x}{r_1}}
\right|^2 du ,}
where $g_1\pr{x}$ is a smooth compactly supported function in $\left[\frac 12, \frac 52\right]$, and $g_2(U) = \min \{ \frac{1}{|U|}, \frac{R}{K}\}.$

From Lemma \ref{lem:asymforpsi} Equation (\ref{psiboundforbigx}), recalling definition of $f(x)$ in (\ref{def:fx}) and $\frac{MNd_2^2d_1^3}{R^4\ell^2} \geq T^{\epsilon_1}$, we have that

\est{F_+\pr{\frac{m\ell^2}{r_1^4\,d_1d_2^2} ; \frac{r_1d_1d_2}{\ell}} &= \frac{m\ell^2}{r_1^4\,d_1d_2^2} \int_0^\infty w_3\pr{\frac yN} \e{\frac{u\ell y}{r_1d_1d_2T}} \sum_{j = 1}^{\mathcal K} \frac{1}{\pr{\frac{m\ell^2y}{r_1^4\,d_1d_2^2}}^{\frac j4 + \frac 18}} \\
& \hskip 0.5in \times  \left[ c_j \e{4\pr{\frac{m\ell^2y}{r_1^4\,d_1d_2^2}}^{\frac 14}} W\pr{8\pi \pr{\frac{m\ell^2y}{r_1^4\,d_1d_2^2}}^{\frac 14}} \right. \\
	&\hskip 1in + \left. d_j \e{ - 4\pr{\frac{m\ell^2y}{r_1^4\,d_1d_2^2}}^{\frac 14}} \overline{W}\pr{8\pi \pr{\frac{m\ell^2y}{r_1^4\,d_1d_2^2}}^{\frac 14}} \right]  \> dy + O\pr{T^{-100}},}
where $\mathcal K$ is sufficiently large.  Without loss of generality, we consider the term $j = 1$ above, the other terms being similar and also visibly smaller. Moreover, since $U$ can be both negative and positive, we can consider the term $\e{4\pr{\frac{m\ell^2y}{r_1^4\,d_1d_2^2}}^{\frac 14}}.$


Hence pulling out factors of $d_1, d_2, r_1, \ell$, using $r_1 \sim \frac{R\ell}{d_1d_2}$ , and opening up the square in \eqref{eqn:J}, we see that we need to bound
\es{\label{eqn:Jafter} g_2(U)  \sum_{d_1 \ll R\ell} \sum_{d_2 \ll R\ell} \frac{1}{d_1^3d_2^2} \left( \frac{d_1^3d_2^2}{R^4\ell^2}\right)^{\frac 54}  \cJ_0(d_1, d_2)}
where we examine
\es{ \label{def:cJ0}
 \cJ_0(d_1, d_2) := \sum_{\substack{r_1 \sim \frac{R\ell}{d_1d_2} }}   \sumstar_{x \bmod{r_1}} \; \sum_{m_1, m_2} \frac{A(m_1, d_2, d_1) \overline{A(m_2, d_2, d_1)}}{m_1^{\frac 38}m_2^{\frac 38}} \e{\frac{m_1x - m_2x}{r_1}}  \, \fI(r_1, x, m_1, m_2; d_1, d_2),
}
and 
\est{
\fI &= \fI(r_1, x, m_1, m_2; d_1, d_2) \\
&:= \int_{-\infty}^{\infty} g_1\pr{\frac uU} \int_{0}^{\infty} \int_0^\infty  \frac{1}{y_1^{\frac 38 }y_2^{\frac 38}}w_3\pr{\frac {y_1}N} w_3\pr{\frac {y_2}N} W\pr{8\pi \pr{\frac{m_1\ell^2 y_1}{r_1^4\,d_1d_2^2}}^{\frac 14}}  W\pr{8\pi \pr{\frac{m_2\ell^2 y_2}{r_1^4\,d_1d_2^2}}^{\frac 14}} \\
& \hskip 1in \times \e{\frac{u\ell(y_1 - y_2)}{r_1d_1d_2T}} \e{\frac{4\ell^{\frac 12}}{r_1d_1^{\frac 14} d_2^{\frac 12}}((m_1y_1)^{\frac 14} - (m_2y_2)^{\frac 14})} \> dy_1 \> dy_2 \> du.
}
By the change of variable from $y_i$ to $y_i N$ and from $u$ to $uU$, and letting 
$$ w_4(y_i) =  \frac{1}{y_i^{\frac 38}} w_3\pr{y_i} \ \  \ \textrm{ and}$$
\es{ \label{def:w2} w_5(y_i, m_i) = w_5(y_i, m_i, r_1, \ell; d_1, d_2) :=  w_4\pr{y_i} W\pr{8\pi \pr{\frac{m_i\ell^2Ny_i}{r_1^4\,d_1d_2^2}}^{\frac 14}},}
$\fI$ can then be written as
\es{\label{def:fI} \fI = N^{\frac 54}U \int_{-\infty}^{\infty} g_1\pr{u} \int_{0}^{\infty} \int_0^\infty&  w_5\pr{y_1, m_1} w_5\pr{y_2, m_2}  \e{\frac{uU\ell N(y_1 - y_2)}{r_1d_1d_2T}} \\
	& \times \e{\frac{4\ell^{\frac 12}N^{\frac 14}}{r_1 d_1^{\frac 14} d_2^{\frac 12}}((m_1y_1)^{\frac 14} - (m_2y_2)^{\frac 14})} \> dy_1 \> dy_2 \> du.}
Note that the functions $w_4, w_5$ are smooth with compact support and satisfy the bound
$$w_4^{(j)} (x) \ll_j 1,  \ \ \ \ \ \ w_5^{(j)}(x, m_i) \ll_j 1.
$$ 

Let $h(y) = \frac{uU \ell N y}{r_1d_1d_2T} + \frac{4N^{\frac 14} \ell^{\frac 12}(my)^{\frac 14}}{r_1d_1^{\frac 14}d_2^{\frac 12}}. $ 
 If $\frac{|U|N}{RT} \ll T^{\epsilon_1/8}$ or $U$ is positive then the second term dominates and we have $h'(y) \gg T^{\epsilon_1/4}$ also.  Integration by parts many times with respect to $y_i$ shows that the contribution from these terms are also negligible.  So it suffices to consider when $U$ is negative and $\frac{|U|N}{RT} \gg T^{\epsilon_1/8}$. Moreover, if $M \geq \frac{N^3\ell^2U^4 }{T^4d_1^3 d_2^2} T^{\epsilon_1}$, then 
$$ |h'(y)| = \left|\frac{uU\ell N}{r_1d_1d_2T} + \frac{N^{\frac 14}\ell^{\frac 12}m^{\frac 14}}{r_1d_1^{\frac 14}d_2^{\frac 12}y^{\frac 34}}\right| \gg \frac{N|U|}{RT} T^{\epsilon_1/4} \gg T^{3\epsilon_1/8}. $$
Again we can do integration many times and derive that the contribution from these terms are negligible. Therefore we restrict our consideration to 
\begin{equation}\label{eqn:Mbdd}
M \leq \frac{N^3\ell^2U^4 }{T^4d_1^3d_2^2} T^{\epsilon_1}.
\end{equation}

Next we consider the integration over $u$ in Equation (\ref{def:fI}). Let $\Delta = y_1 - y_2$ so by the change of variable $y_1 = \Delta + y_2$, we obtain that the integration in Equation (\ref{def:fI})  becomes 
\est{ \fI = N^{\frac 54}U \int_0^\infty \int_{-y_2}^{\infty} &  w_5\pr{y_2 + \Delta, m_1} w_5\pr{y_2, m_2}  \widehat{g_1} \bfrac{-U\ell N \Delta}{r_1d_1d_2T} \\
	& \times \e{\frac{4\ell^{\frac 12}N^{\frac 14}}{r_1 d_1^{\frac 14} d_2^{\frac 12}}(m_1^{\frac 14}(y_2 + \Delta)^{\frac 14} - (m_2y_2)^{\frac 14})} \> d\Delta \> dy_2,}
where $\widehat{g_1}$ is the usual Fourier transform of $g_1$ and is Schwartz class.  Thus,
$$ \widehat{g_1} \bfrac{-U\ell N \Delta}{r_1d_1d_2T} \ll T^{-100}$$ whenever 
$$ |\Delta| \ge T^{\epsilon} \frac{RT}{|U|N} $$
for any $\epsilon>0$.  We thus assume that
\begin{equation}\label{eqn:Deltabdd}
|\Delta| \le T^{\epsilon} \frac{RT}{|U|N} \le T^{-\epsilon_1/8 + \epsilon}.
\end{equation}

Applying the Taylor expansion for $(y_2 + \Delta)^{\frac 14}$, we have that 
$$ (y_2 + \Delta)^{\frac 14} = y_2^{\frac 14}\left(1 + \mathcal P(\Delta, y_2) \right)$$
where $\mathcal P(\Delta, y_2) = \sum_{j = 1}^\infty c_j \left( \frac{\Delta}{y_2}\right)^j.$  Also let $\mathcal K = \mathcal K(R, T, N, U)$ be an interval such that
\es{\label{def:K} \mathcal K = \left\{  \Delta \geq -y_2 :  \, |\Delta| \le T^{\epsilon} \frac{RT}{|U|N} \right\}}
 Thus
\es{\label{eqn:fI2} \fI = N^{\frac 54}U  \int_0^\infty \int_{\mathcal K} &  w_5\pr{y_2 + \Delta, m_1} w_5\pr{y_2, m_2}   \widehat{g_1} \bfrac{-U\ell N \Delta}{r_1d_1d_2T} \\
	& \times \e{\frac{4N^{\frac 14}\ell^{\frac 12}y_2^{\frac 14}}{r_1d_1^{\frac 14}d_2^\frac{1}{2}}(m_1^{\frac 14} - m_2^{\frac 14}) } \e{ \frac{4(N\ell^2y_2m_1)^{\frac 14}}{r_1d_1^{\frac 14}d_2^\frac{1}{2}} \mathcal P(\Delta, y_2)} \> d\Delta \> dy_2.}

Since $\Delta \ll T^{-3\epsilon_1/32 }$ (picking $\epsilon = \epsilon_1/32$ in \eqref{eqn:Deltabdd}) and $\frac{|U\Delta|N}{RT} \ll T^{\epsilon_1/32}$, 
\begin{equation}\label{eqn:PDeltabdd}
\frac{4(N\ell^2y_2m_1)^{\frac 14}}{r_1d_1^{\frac 14}d_2^\frac{1}{2}} \mathcal P(\Delta, y_2)  \ll \frac{|U|NT^{\epsilon_1/4}}{RT}\mathcal P(\Delta, y_2) \ll T^{9\epsilon_1/32}.
\end{equation}

If $\frac{4N^{\frac 14}\ell^{\frac 12}y_2^{\frac 14}}{r_1d_1^{\frac 14}d_2^\frac{1}{2}}|m_1^{\frac 14} - m_2^{\frac 14}|  \gg T^{5\epsilon_1/16} $, then we can do integration with respect to $y_2$ many times and obtain that the contribution of these terms is negligible. Therefore we consider when
\begin{equation*}
|m_1^{\frac 14} - m_2^{\frac 14}|  \ll T^{5\epsilon_1/16} \frac{r_1d_1^{\frac 14}d_2^{\frac 12}}{N^{\frac 14} \ell^{\frac 12}} \asymp T^{5\epsilon_1/16} \cR M^{1/4},
\end{equation*}where
\begin{equation}\label{eqn:cRdef}
\cR = \frac{R\ell^{\frac 12}}{N^{\frac 14}M^{\frac 14}d_1^{\frac 34} d_2^{\frac 12}} \ll T^{-\epsilon_1/4},
\end{equation} upon recalling that we are working in the range $M \ge \frac{R^4\ell^2}{Nd_1^3d_2^2}T^{\epsilon_1}$.  Thus
\es{\label{boundfordiffm} |m_1 - m_2| \ll T^{5\epsilon_1/16} \cR M := \mathcal L,}and \eqref{eqn:PDeltabdd} becomes
\begin{equation}\label{eqn:PDeltabdd2}
\mathcal P(\Delta, y_2) \ll \cR T^{9\epsilon_1/32}.
\end{equation}

Next, we divide the range for $m_1, m_2$ into intervals $\mathcal C_{\eta_1}$ and $\mathcal C_{\eta_2}$ of length $T^{-\epsilon_1}\mathcal L$, where $\eta_1$ and $\eta_2$ are the left endpoints of the intervals $\mathcal C_{\eta_1}$ and $\mathcal C_{\eta_2}$ respectively. When $m_i \in \mathcal C_{\eta_i}$ for some $\eta_i$ and Equation (\ref{boundfordiffm}) holds,  the restriction of the length of the intervals implies that $|\eta_1 - \eta_2| \ll \mathcal L.  $ Hence for fixed $\mathcal C_{\eta_1}$, there are $O(T^{\epsilon_1})$ choices for $\mathcal C_{\eta_2}$, and so there are $O\left(\frac{1}{\cR} T^{27\epsilon_1/16} \right)$ relevant pairs of intervals $(\mathcal C_{\eta_1}, \mathcal C_{\eta_2})$ with end points satisfying Equation (\ref{boundfordiffm}).  We let $\sumt_{(\mathcal C_{\eta_1}, \mathcal C_{\eta_2})}$ denote the sum over such pairs.

From Equations (\ref{def:cJ0}), (\ref{def:w2}), (\ref{eqn:fI2}), and the trivial bound $\widehat{g_1}\pr{\frac{-U\ell N\Delta}{r_1d_1d_2T}} \ll 1$, we have that 
\es{\label{eqn1:boundcJ0}\cJ_0 \ll g_2(U) U N^{\frac 54} \int_0^\infty \int_{\mathcal K} & | w_4\pr{y_2 + \Delta} w_4\pr{y_2} | \ \sumt_{(\mathcal C_{\eta_1}, \mathcal C_{\eta_2})}   \mathcal S(\mathcal C_{\eta_1}, \mathcal C_{\eta_2}) d\Delta dy_2 , }
where
\est{S(\mathcal C_{\eta_1}, \mathcal C_{\eta_2})  := \sum_{r_1 \sim \frac{R\ell}{d_1d_2}} \sumstar_{x \bmod r_1}  |\mathcal S_1(C_{ \eta_1}) \mathcal S_2(C_{\eta_2})| ;}
\est{\mathcal S_1(\mathcal C_{\eta_1}) &:= \sum_{m_1 \in \mathcal C_{\eta_1}} \frac{A(m_1, d_2, d_1) }{m_1^{\frac 38}} W\pr{8\pi \pr{\frac{m_1N\ell^2(y_2 + \Delta)}{r_1^4\,d_1d_2^2}}^{\frac 14}}\e{\frac{4N^{\frac 14}\ell^{\frac 12}y_2^{\frac 14} m_1^{\frac 14}}{r_1d_1^{\frac 14}d_2^\frac{1}{2} }} \\
	&\hskip 3in \times \e{ \frac{4(N\ell^2y_2m_1)^{\frac 14}}{r_1d_1^{\frac 14}d_2^\frac{1}{2}} \mathcal P(\Delta, y_2)} \e{\frac{m_1x}{r_1}};}
and 

\est{\mathcal S_2(\mathcal C_{\eta_1}) &:= \sum_{m_2 \in \mathcal C_{\eta_2}} \frac{\overline{A(m_2, d_2, d_1)} }{m_2^{\frac 38}} W\pr{8\pi \pr{\frac{m_2N\ell^2y_2 }{r_1^4\,d_1d_2^2}}^{\frac 14}}\e{-\frac{4N^{\frac 14}\ell^{\frac 12}y_2^{\frac 14} m_2^{\frac 14}}{r_1d_1^{\frac 14}d_2^\frac{1}{2} }} \e{-\frac{m_2x}{r_2}}.}

By the inequality $2|ab| \leq |a|^2 + |b|^2$. We have that 
\begin{equation}\label{eqn:bound2}
\left| S(\mathcal C_{\eta_1}, \mathcal C_{\eta_2}) \right| \leq \frac{1}{2}  \sum_{r \sim \frac{R\ell}{d_1d_2}} \sumstar_{x \bmod r_1}  |\mathcal S_1(C_{ \eta_1})|^2  + \frac 12 \sum_{r \sim \frac{R\ell}{d_1d_2}} \sumstar_{x \bmod r_1} |\mathcal S_2(C_{\eta_2})|^2. 
\end{equation}

To bound $\cJ_0$ in Equation (\ref{eqn1:boundcJ0}), we first prove the following Lemma.
\begin{lem} \label{lem:SiCi} Let $\mathcal S_i(C_{\eta_i})$ for $i = 1, 2$ be defined as above. We have
	$$ \sum_{r_1 \sim \frac{R\ell}{d_1d_2}} \sumstar_{x \bmod r_1}  |\mathcal S_i(C_{ \eta_i})|^2 \ll M^{-3/4} \left( \left( \frac{R\ell}{d_1d_2}\right)^2 + \mathcal L \right)  \fS(\eta_i),$$ where
\begin{equation}\label{eqn:fsdef}
\fS(\eta_i) := \sum_{m \in \mathcal C_{\eta_i}} |A(m, d_2, d_1)|^2.
\end{equation}

\end{lem}

\begin{proof}
	We will bound $\sum_{r_1 \sim \frac{R\ell}{d_1d_2}} \sumstar_{x \bmod r_1}  |\mathcal S_1(C_{ \eta_1})|^2$  since the proof for $\mathcal S_2$ is similar and somewhat simpler due to the absence of dependence on $\Delta$. 
	
	Define 
	$$ \mathcal A(x) =   W\pr{8\pi\pr{y_2+\Delta }^{\frac 14} x}\e{4y_2^{\frac 14}x}\e{-4\mathcal P(\Delta, y_2)y_2^{\frac 14} x}.$$
	We will take a Taylor expansion of $\mathcal A\pr{\pr{\frac{N\ell^2m_1}{r_1^4 \,d_1d_2^2}}^{\frac 14}}$ around 
	\begin{equation}\label{eqn:xdef}
	x = \pr{\frac{N \ell^2\eta_1}{r_1^4 \,d_1d_2^2}}^{\frac 14} \asymp \frac{1}{\cR}
	\end{equation} to separate variables $r_1$ and $m_1$ before applying the Large Sieve. We write
	\est{ \mathcal A\pr{\pr{\frac{N\ell^2m_1}{r_1^4 \,d_1d_2^2}}^{\frac 14}} = \sum_{\alpha = 0}^\infty  \frac{1}{\alpha!} \mathcal A^{(\alpha)}\pr{\pr{\frac{N \ell^2\eta_1}{r_1^4 \,d_1d_2^2}}^{\frac 14}}\, \pr{\pr{\frac{N \ell^2 m_1}{r_1^4 \,d_1d_2^2}}^{\frac 14} -\pr{\frac{N \ell^2\eta_1}{r_1^4 \,d_1d_2^2}}^{\frac 14}  }^{\alpha}. }
	
By Equation (\ref{boundfordiffm}), $m_1, \eta_1 \sim M$ and the Mean Value Theorem, we obtain that 
\begin{equation}\label{eqn:metabdd}
 |m_1^{\frac 14} - \eta_1^{\frac 14}| \leq \frac{|m_1 - \eta_1|}{M^{\frac 34}} \ll \frac{\cR M T^{\frac{5\epsilon_1}{16}}}{M^{\frac 34}T^{\epsilon_1}} \ll \cR M^{1/4} T^{-\frac{11\epsilon_1}{16}} .
\end{equation}
Moreover by the property of $ W(x)$ (e.g. see \cite{Watt} p.206), Equations \eqref{eqn:xdef} and \eqref{eqn:PDeltabdd2}, and $y_2 \ll 1$, 
$$ \mathcal A^{(\alpha)}(x) \ll  \alpha! c^\alpha x^{-\alpha} + \left(\cR T^{9\epsilon_1/32}\right)^\alpha \ll \alpha! c^\alpha \left(\cR T^{9\epsilon_1/32}\right)^{\alpha}$$
where $c$ is some absolute constant. Thus
$$ \frac{1}{\alpha!} \mathcal A^{(\alpha)}\pr{\pr{\frac{N\ell^2\eta_1}{r_1^4 \,d_1d_2^2}}^{\frac 14}}\, \pr{\pr{\frac{N \ell^2m_1}{r_1^4 \,d_1d_2^2}}^{\frac 14} -\pr{\frac{N \ell^2\eta_1}{r_1^4 \,d_1d_2^2}}^{\frac 14}  }^{\alpha} \ll \pr{ \cR \frac{c}{T^{\frac{13\epsilon_1}{32}}}}^\alpha.$$
Recalling \eqref{eqn:cRdef}, we can choose $B$ such that  $\pr{\cR \frac{c}{T^{\frac{13 \epsilon_1}{32}}}}^\alpha \ll T^{-100}$ and obtain that
$$ \mathcal A\pr{\pr{\frac{N \ell^2m_1}{r_1^4 \,d_1d_2^2}}^{\frac 14}} = \sum_{0 \leq \alpha \leq B}  \frac{1}{\alpha!} \mathcal A^{(\alpha)}\pr{\pr{ \frac{N \ell^2\eta_1}{r_1^4 \,d_1d_2^2}}^{\frac 14}}\, \pr{\pr{\frac{N \ell^2m_1}{r_1^4 \,d_1d_2^2}}^{\frac 14} -\pr{\frac{N \ell^2\eta_1}{r_1^4 \,d_1d_2^2}}^{\frac 14}  }^{\alpha} + O(T^{-100}).$$
Hence to bound $\sum_{r_1 \sim \frac{R\ell}{d_1d_2}} \sumstar_{x \bmod r_1}  |\mathcal S_1(C_{ \eta_1})|^2$, it is enough bound $\mathcal S_1(C_{\eta_1}; \alpha)$ for fixed $\alpha$, which is defined to be
\es{\label{eqn:bound1} 
&\sum_{r_1 \sim \frac{R\ell}{d_1d_2}} \sumstar_{x \bmod r_1} \left| \sum_{m_1 \in \mathcal C_{\eta_1}} \frac{\overline{A(m_1, d_2, d_1)} }{m_1^{\frac 38}} \frac{1}{\alpha!} \mathcal A^{(\alpha)}\pr{\pr{ \frac{N \ell^2\eta_1}{r_1^4 \,d_1d_2^2}}^{\frac 14}}\,\right. \\
& \hskip 2in \times \left. \pr{\pr{\frac{N \ell^2m_1}{r_1^4 \,d_1d_2^2}}^{\frac 14} -\pr{\frac{N \ell^2\eta_1}{r_1^4 \,d_1d_2^2}}^{\frac 14}  }^{\alpha}  \e{-\frac{m_1x}{r_1}} \right|^2 \\
&\ll c^{\alpha} \bfrac{N^{1/2}d_1^{3/2}d_2}{R^2 \ell}^\alpha \left(\cR T^{9\epsilon_1/32}\right)^{2\alpha}\sum_{r \sim \frac{R\ell}{d_1d_2}} \sumstar_{x \bmod r_1} \left| \sum_{m_1 \in \mathcal C_{\eta_1}} \frac{\overline{A(m_1, d_2, d_1)} }{m_1^{\frac 38}} \, \pr{m_1^{\frac 14} -\eta_1^{\frac 14}}^{\alpha} \e{-\frac{m_1x}{r_1}} \right|^2  \\
&\ll c^{\alpha} \bfrac{N^{1/2} d_1^{3/2}d_2\cR^2 T^{9\epsilon_1/16}}{R^2 \ell}^{\alpha} \left( \left(\frac{R\ell}{d_1d_2}\right)^2 + \mathcal L \right)\left(\sum_{m \in \mathcal C_{\eta_1}} \frac{|A(m, d_2, d_1)|^2 }{m^{\frac 34}} \left|m^{\frac 14} -\eta_1^{\frac 14}\right|^{2 \alpha}\right)  \\
&= M^{-3/4} \bfrac{c N^{1/2}d_1^{\frac 32} d_2 M^{1/2} \cR^2 \cR^2 }{R^2 \ell T^{13\epsilon_1/16}}^{\alpha} \left( \left(\frac{R\ell}{d_1d_2}\right)^2 + \mathcal L \right)  \fS(\eta_1)
}
by \eqref{eqn:metabdd} and where $\fS$ is as defined in \eqref{eqn:fsdef}.  Using that $\cR^2 \asymp \frac{R^2 \ell}{(MN)^{1/2}d_1^{\frac 32}d_2}$ and $\cR \ll T^{-\epsilon_1/4}$ (see Equation \ref{eqn:cRdef}), we obtain that the quantity in \eqref{eqn:bound1} is maximized when $\alpha = 0$ for sufficiently large $T$, so that 
\begin{equation}
\mathcal S_1(C_{\eta_1}; \alpha)  \ll M^{-3/4} \left( \left(\frac{R\ell}{d_1d_2}\right)^2 + \mathcal L \right)  \fS(\eta_1),
\end{equation}for all $\alpha$.

\end{proof}
By Lemma \ref{lem:SiCi} and \eqref{eqn:bound2},
\begin{align*}
\sumt_{(\mathcal C_{\eta_1}, \mathcal C_{\eta_2})}   \mathcal S(\mathcal C_{\eta_1}, \mathcal C_{\eta_2})
&\ll  M^{-3/4} \left( \left(\frac{R\ell}{d_1d_2}\right)^2 + \mathcal L \right)  \sumt_{(\mathcal C_{\eta_1}, \mathcal C_{\eta_2})} (\fS(\eta_1) + \fS(\eta_2))\\
&\ll M^{-3/4} \left( \left(\frac{R\ell}{d_1d_2}\right)^2 + \mathcal L \right) T^{\epsilon_1} \sum_{m\sim M} |A(m, d_2, d_1)|^2,
\end{align*}since for each $\eta_1$ there are $\ll T^{\epsilon_1}$ choices for $\eta_2$ and vice versa.  In the sequel, we shall simplify notation slightly and write our bounds in terms of $\epsilon>0$ which can be made arbitrarily small by taking $\epsilon_1$ sufficiently small.  Using the bound in Lemma, which is \ref{lem:ramanujanonaverage}
$$\sum_{m\sim M} |A(m, d_2, d_1)|^2\ll (Md_1d_2)^{1+\epsilon} \ll M d_1d_2T^{\epsilon},
$$
and $\mathcal L = \mathcal R M T^{5\epsilon_1/16}$,
we see that the above is
$$ \sumt_{(\mathcal C_{\eta_1}, \mathcal C_{\eta_2})}   \mathcal S(\mathcal C_{\eta_1}, \mathcal C_{\eta_2}) \ll M^{1/4} d_1d_2 T^{\epsilon} \left( (R\ell)^2 + \cR M \right).
$$
It follows from \eqref{eqn1:boundcJ0} and \eqref{eqn:Deltabdd} that
\begin{align*}
\cJ_0 \ll M^{1/4}g_2(U) U N^{\frac 54}d_1d_2 T^{\epsilon} \frac{RT}{|U|N}  \left( \left(\frac{R\ell}{d_1d_2}\right)^2 + \cR M \right),
\end{align*}and using that $g_2(U) U \ll 1$ for all $U$, and referring to Equations \eqref{eqn:sum3} and (\ref{eqn:Jafter}), it now suffices to bound
\es{\label{eqn:bound3}
&\ell^2 \frac{T^{1 + \epsilon}R^3}{N}  N^{\frac 54} \frac{RT}{|U|N} \sum_{d_1 \ll R\ell} \sum_{d_2 \ll R\ell} \frac{d_1d_2}{d_1^3d_2^2} \left( \frac{d_1^3d_2^2}{R^4\ell^2}\right)^{\frac 54} \left( \left(\frac{R\ell}{d_1d_2}\right)^2 + \cR M \right) M^{\frac 14} \\
&\ll \frac{T^{2 + \epsilon}}{RN^{3/4}\ell^{1/2}|U|}   \sum_{d_1 \ll R\ell} \sum_{d_2 \ll R\ell} d_1d_2(d_1^3d_2^2)^{\frac 14}  \left( \left(\frac{R\ell}{d_1d_2}\right)^2 + \frac{R\ell^{\frac 12}}{N^{\frac 14}M^{\frac 14} d_1^{\frac 34}d_2^{\frac 12}} M \right)M^{\frac 14}  \\
&\ll \frac{T^{2 + \epsilon}}{RN^{3/4}\ell^{1/2}|U|}  \sum_{d_1 \ll R\ell} \sum_{d_2 \ll R\ell}   \left( \frac{M^{1/4}}{d_1^{\frac 14} d_2^{\frac 12}} (R\ell)^2 + \frac{R\ell^{1/2} d_1d_2}{N^{1/4}} M \right),
}
by \eqref{eqn:cRdef}.  Using that $M \ll \frac{N^{3}\ell^2U^4}{T^4d_1^3d_2^2} T^{\epsilon_1}$ from \eqref{eqn:Mbdd} and summing over $d_1, d_2$, we obtain that the quantity in \eqref{eqn:bound3} is 
\begin{align}\label{eqn:bound4}
&\ll \frac{T^{2 + \epsilon}}{RN^{3/4}\ell^{1/2}|U|} \left( R^2\ell^{5/2} \frac{N^{3/4}|U|}{T} + \frac{RU^4 N^{3}\ell^{5/2}}{N^{1/4}T^4}\right)
\ll T^{1 + \epsilon}\ell^{2} R + T^{\epsilon} \frac{N^2\ell^2}{T^2},
\end{align}where we have used that $|U| \ll T^\epsilon$.  Now using that $R \ll \frac{N}{T}$ and $N\ll \frac{T^{2+\epsilon}}{\ell^2}$, we see that \eqref{eqn:bound4} is $\ll T^{2+\epsilon}$ as desired.

\appendix
\section{Proof of Lemma \ref{lem:asymforpsi}} \label{sec:proofofasympPsi}
The proof requires properties of $J$-Bessel function as the following.
\begin{lem} \label{lem:Besselresult}
	For any integer $k \geq 0,$ 
	\es{\label{asympJxbig} J_{k} (2\pi x) = \frac{1}{2\pi\sqrt x}\pg{ W_k(2\pi x)\e{x - \frac k4 - \frac 18}  +  \overline{W_k}(2\pi x)\e{-x + \frac k4 + \frac 18}},}
	where $W_k^{(j)}(x) \ll_{j, k} x^{-j} $.  Moreover, 
	\es{\label{asympJxSm} J_{k} (2 x) = \sum_{\ell = 0}^{\infty} (-1)^{\ell} \frac{x^{2\ell + k }}{\ell! (\ell + k )!}.}
	
\end{lem}
These results are standard, e.g. see \cite{Watt} for the proof.
\begin{proof}[Proof of Lemma \ref{lem:asymforpsi}]
	We start with evaluating the asymptotic formula for $\Psi_{+}$. From the definition of $\Psi_+$ in Equation (\ref{def:Psipm}), for $x > 0$  
	\est{\label{eqn:F+}\Psi_+(x) &= \frac{1}{2\pi i} \int_{(-\sigma)}  \frac{\pi^{4s - 2} \Gamma\pr{\frac{1-s - \alpha_1}{2}}\Gamma\pr{\frac{1-s - \alpha_2}{2}}\Gamma\pr{\frac{1-s - \alpha_3}{2}}\Gamma\pr{\frac{1-s - \alpha_4}{2}}}{\Gamma\pr{\frac{s + \alpha_1}{2}}\Gamma\pr{\frac{s + \alpha_2}{2}}\Gamma\pr{\frac{s + \alpha_3}{2}}\Gamma\pr{\frac{s + \alpha_4}{2}}}\widetilde{\psi}(s) x^{s} \> ds \\
		&= \frac{2x \pi^2}{2\pi i} \int_{(\sigma_1)}  \frac{\pi^{-8s}x^{-2s} \Gamma\pr{s- \frac{ \alpha_1}{2}}\Gamma\pr{s- \frac{ \alpha_2}{2}}\Gamma\pr{s - \frac{ \alpha_3}{2}}\Gamma\pr{s-\frac{\alpha_4}{2}}}{\Gamma\pr{\frac 12 -s +\frac{ \alpha_1}{2}}\Gamma\pr{\frac 12 - s + \frac{  \alpha_2}{2}}\Gamma\pr{\frac 12 - s +\frac{ \alpha_3}{2}}\Gamma\pr{\frac 12 - s + \frac{ \alpha_4}{2}}}\widetilde{\psi}(-2s + 1)  \> ds,}
	where $\sigma_1 = \frac{1 + \sigma}{2} > \frac{15}{68}$ \footnote{This is due to the bound for $\lambda_j$ in the work of Luo, Rudnick and Sarnak \cite{LRS}}. 
	Let 
	$$ H(s) = 4^{8s - 2} \frac{\prod_{j = 1}^4\Gamma\pr{s - \frac{\alpha_j}{2}} \Gamma\pr{\frac 12 - 4s}}{\prod_{j = 1}^4 \Gamma\pr{\frac 12 - s + \frac{\alpha_j}{2}} \Gamma\pr{4s - \frac 32}} - 1.$$
	From the Stirling's formula (e.g. \cite{Li}) below
	\est{\log \Gamma(s + c) = \pr{s + c  - \frac 12} \log s - s + \frac 12 \log 2\pi + \sum_{j = 1}^{\mathcal K} \frac{a_j}{s^j} + O_\delta \pr{\frac{1}{|s|^{K +1}}},}
	where $c$ is a constant, $a_j$ are suitable constants, $K$ is a fixed positive integer, $|\arg(s)| \leq \pi - \delta$, $\delta > 0$, excluding the points $s = 0$ and the neighborhoods of the poles of $\Gamma(s + c)$, it can be shown that 
	\es{\label{expansionH}H(s) = \sum_{j = 1}^{\mathcal K} \frac{b_j}{s^j} +  O \pr{\frac{1}{|s|^{\mathcal K +1}}},}
	where $b_j$ are appropriate constants, depending on $\alpha_i.$ 
	
	By the definition of $\Psi_+(x)$, we have that 
	\est{\Psi_+(x)  
		&= \frac{2x \pi^2}{2\pi i} \int_{(\sigma_1)}  \frac{\pi^{-8s}4^{2 - 8s}x^{-2s} \Gamma\pr{4s -  \frac{ 3}{2}}}{\Gamma\pr{\frac 12 -4s}}\widetilde{\psi}(-2s + 1)  \> ds  \\
		& \hskip 0.7in + \frac{2x \pi^2}{2\pi i} \int_{(\sigma_1)} \frac{\pi^{-8s}4^{2 - 8s}x^{-2s} \Gamma\pr{4s -  \frac{ 3}{2}}}{\Gamma\pr{\frac 12 -4s}}H(s)\widetilde{\psi}(-2s + 1)  \> ds \\
		& =: I_1 + I_2. }
	Firstly, let us consider $I_1$. By changing  variables $4s - 3/2 $ to $w$, we obtain that 
	\est{I_1 &= \frac{x \pi^2}{4\pi i} \int_{(\sigma_2)}  \frac{\pi^{-2w - 3}4^{-2w - 1}x^{-\frac w2 - \frac 34} \Gamma\pr{w}}{\Gamma\pr{-1 - w}}\widetilde{\psi}\pr{-\frac w2 + \frac 14}  \> dw \\
		&= \frac{x} {16\pi^2 i x^{\frac 34}} \int_{(\sigma_2)}  \frac{ \Gamma\pr{w}}{\Gamma\pr{-1 - w}} (4\pi)^{-2w}x^{-\frac w2 } \widetilde{\psi}\pr{-\frac w2 + \frac 14}  \> dw	, }
	where $\sigma_2 = 4\sigma_1 - \frac 32 > - \frac{21}{34}$. We can choose $\sigma_2 > 0$.  We move countour integration to the left to Re$(s) = -\infty$, picking up poles of $\Gamma(w)$ at $w = -2, -3, .....$. Note that there are zeros at $w = 0, -1$ from $\Gamma(-1-w)$, which are cancelled with the poles at $0, -1$. Hence 
	\est{I_1 &= \frac{x} {8\pi  x^{\frac 34}}  \sum_{n = 2}^{\infty} \frac{(-1)^n}{n! \Gamma(-1 + n)} (4\pi x^{\frac 14} )^{2n} \widetilde{\psi}\pr{\frac n2 + \frac 14} \\
		&= \frac{x} {8\pi  x^{\frac 34}}  \int_{0}^{\infty} \psi(y) \sum_{n = 2}^{\infty} \frac{(-1)^n}{n! (n-2)!} (4\pi x^{\frac 14} )^{2n} y^{\frac n2 + \frac 14 - 1 } \> dy \\
		&= 2\pi x  \int_{0}^{\infty} \frac{\psi(y)}{(xy)^{\frac 14}} \sum_{n = 0}^{\infty} \frac{(-1)^n}{(n + 2)! n!} (4\pi (xy)^{\frac 14} )^{2n + 2} \> dy \\
		& = 2\pi x \int_0^\infty \frac{\psi(y)}{(xy)^{\frac 14}} J_2\pr{8\pi(xy)^{\frac 14} } \> dy}
	where the last equation comes from the representation of $J$-Bessel function in Equation (\ref{asympJxSm}). 
	By Lemma \ref{lem:Besselresult} Equation (\ref{asympJxbig}), we then have 
	$$ I_1 =  x \int_0^\infty \frac{\psi(y)}{(xy)^{\frac 14 + \frac 18}} \left[ c \e{4(xy)^{\frac 14}} W_2(8\pi (xy)^{\frac 14}) + d \e{ - 4(xy)^{\frac 14}} \overline{W_2}(8\pi (xy)^{\frac 14}) \right]  \> dy  $$
	for some constants $c, d$.

	Next we consider $I_2$. Since the expansion of $H(s)$ is of the form in (\ref{expansionH}). Therefore it is sufficient to consider 
	\est{I_{2, j} = \frac{2x \pi^2}{2\pi i} \int_{(\sigma_2)} \frac{\pi^{-8s}4^{2 - 8s}x^{-2s} \Gamma\pr{4s -  \frac{ 3}{2}}}{\Gamma\pr{\frac 12 -4s}s^j}\widetilde{\psi}(-2s + 1)  \> ds. }
	By the change of variables $4s - 3/2 \rightarrow w$, we have
	\est{ I_{2, j} = \frac{4^j x }{16\pi^2 ix^{\frac 34}} \int_{(\sigma_1)} \frac{\Gamma\pr{w}}{\Gamma\pr{-1 - w} (w + 3/2)^j}(4\pi)^{-2w}x^{-\frac w2 } \widetilde{\psi}\left(-\frac w2 + \frac 14\right)  \> dw .}
	
	We illustrate how to find asymptotic formula for $I_{2, j}$ when $j = 1$. Other cases are proceeded similarly.
	
	\est{I_{2, 1} &=  -\frac{ x }{4\pi^2 ix^{\frac 34}} \int_{(\sigma_2)} \frac{\Gamma\pr{w}}{\Gamma\pr{-1 - w} (-1 - w)}(4\pi)^{-2w}x^{-\frac w2 } \widetilde{\psi}\left(-\frac w2 + \frac 14\right)  \> dw \\
		&+  \frac{ x }{8\pi^2 ix^{\frac 34}} \int_{(\sigma_2)} \frac{\Gamma\pr{w}}{\Gamma\pr{-1 - w} (-1 - w)(w + 3/2)}(4\pi)^{-2w}x^{-\frac w2 } \widetilde{\psi}\left(-\frac w2 + \frac 14\right)  \> dw \\
		&=: I_{2, 1}^1 + I_{2, 1}^2.}
	
	By moving the contour integral to the far left for $I_{2, 1}^1$, we pick up poles at $n = -1, -2, ....$ and obtain that 
	\est{I_{2, 1}^1 &=  -\frac{ x }{2\pi x^{\frac 34}} \sum_{n = 1}^\infty \frac{(-1)^n}{n! \Gamma(n)} (4\pi x^{\frac 14})^{2n} \widetilde{\psi}\pr{\frac n2 + \frac 14} 
		 = -2 x \int_0^\infty \frac{\psi(y)}{(xy)^{\frac 12}} J_1\pr{8\pi(xy)^{\frac 14} } \> dy.}
	
	Next we write $I_{2, 1}^2$ as $I_{2,1}^3 + I_{2,1}^4$, where 
	\est{I_{2, 1}^3 & = -\frac{ x }{8\pi^2 ix^{\frac 34}} \int_{(\sigma_2)} \frac{\Gamma\pr{w}}{\Gamma\pr{- w} (-w)}(4\pi)^{-2w}x^{-\frac w2 } \widetilde{\psi}\left(-\frac w2 + \frac 14\right)  \> dw \\
	I_{2,1}^4 &= \frac{3 x }{16\pi^2 ix^{\frac 34}} \int_{(\sigma_2)} \frac{\Gamma\pr{w}}{\Gamma\pr{- w} (- w)(w + 3/2)}(4\pi)^{-2w}x^{-\frac w2 } \widetilde{\psi}\left(-\frac w2 + \frac 14\right)  \> dw.}
Repeating the above arguments, we have that
\est{I_{2, 1}^3 =  - \frac{x}{4\pi} \int_0^\infty \frac{\psi(y)}{(xy)^{\frac 34}} J_0\pr{8\pi(xy)^{\frac 14} } \> dy.}
We use Equation (\ref{asympJxbig}) to write the $J$-Bessel function in terms of $W_k$. Finally we then apply the same arguments to $I_{2, 1}^4$ and so on, and obtain lower order terms.

\end{proof}

\section*{Acknowledgement} 
We would like to thank Matt Young for enlightening correspondence and particularly for the  helpful reference regarding the exterior square $L$-functions \cite{Kon} and the idea of using a Mobius inversion type formula as in \cite{LiYoung} in the proofs in Section \ref{sec:lemsumpf}.  We also thank the anonymous referees for helpful comments and suggestions.

Xiannan Li would like to acknowledge support from a Simons Collaboration grant and a KSU startup grant. Vorrapan Chandee would like to acknowledge support from AMS-Simons Travel grant.

\end{document}